\documentclass{amsart}
\usepackage{amsthm}
\usepackage{amsmath}
\usepackage{amssymb}
\usepackage{enumerate}
\usepackage{microtype}
\usepackage{graphicx}
\usepackage[hidelinks,pagebackref,pdftex]{hyperref}
\usepackage{booktabs}
\usepackage{color}
\usepackage[dvipsnames]{xcolor}
\usepackage{import}
\usepackage{tikz-cd}

\usepackage[scaled=0.9]{sourcecodepro} 

\AtBeginDocument{%
   \def\MR#1{}
}

\renewcommand*{\backref}[1]{}
\renewcommand*{\backrefalt}[4]{
  \ifcase #1
  [No citations.]
  \or [#2]
  \else [#2]
  \fi }

\makeatletter
\let\c@equation\c@subsection
\makeatother
\numberwithin{equation}{section}

\makeatletter
\let\c@figure\c@equation
\makeatother
\numberwithin{figure}{section}

\theoremstyle{plain}
\newtheorem{theorem}[equation]{Theorem}

\newtheorem{lemma}[equation]{Lemma}

\newtheorem{proposition}[equation]{Proposition}

\theoremstyle{definition}
\newtheorem{definition/}[equation]{Definition}

\newtheorem{example/}[equation]{Example}
\newtheorem{fact/}[equation]{Fact}
\newtheorem{question/}[equation]{Question}
\newtheorem*{question*}{Question}
\newtheorem*{answer*}{Answer}
\newtheorem*{application*}{Application}

\newenvironment{definition}
  {%
   \pushQED{\qed}\begin{definition/}}
  {\popQED\end{definition/}}

\newenvironment{fact}
  {%
   \pushQED{\qed}\begin{fact/}}
  {\popQED\end{fact/}}

\theoremstyle{remark}
\newtheorem{remark/}[equation]{Remark}
\newtheorem*{remark*}{Remark}

\newtheorem{case}{Case}
\newtheorem*{case*}{Case}

\newtheorem*{claim*}{Claim}

\newenvironment{remark}
  {%
   \pushQED{\qed}\begin{remark/}}
  {\popQED\end{remark/}}

\newcommand{\refsec}[1]{Section~\ref{Sec:#1}}

\newcommand{\refthm}[1]{Theorem~\ref{Thm:#1}}

\newcommand{\reflem}[1]{Lemma~\ref{Lem:#1}}
\newcommand{\refprop}[1]{Proposition~\ref{Prop:#1}}

\newcommand{\reffac}[1]{Fact~\ref{Fac:#1}}

\newcommand{\reffig}[1]{Figure~\ref{Fig:#1}}

\newcommand{\refdef}[1]{Definition~\ref{Def:#1}}

\newcommand{\refcase}[1]{Case~\ref{Case:#1}}

\newcommand{\fakeenv}{} 

\newenvironment{restate}[2]  
{
 \renewcommand{\fakeenv}{#2} 
 \theoremstyle{plain}
 \newtheorem*{\fakeenv}{#1~\ref{#2}} 
 \begin{\fakeenv}
}
{
 \end{\fakeenv}
}

\newcommand{\from}{\colon} 
\newcommand{\cross}{\times}
\newcommand{\thsup}{\textrm{th}}

\newcommand{\RR}{{\mathbb{R}}}
\newcommand{\ZZ}{{\mathbb{Z}}}

\newcommand{\calB}{{\mathcal{B}}}

\newcommand{\calF}{{\mathcal{F}}}
\newcommand{\calH}{{\mathcal{H}}}

\newcommand{\calP}{{\mathcal{P}}}

\newcommand{\calT}{{\mathcal{T}}}

\newcommand{\SO}{\operatorname{SO}} 
\newcommand{\RP}{\mathbb{RP}} 
\newcommand{\orb}{{\operatorname{orb}}}

\newcommand{\bdy}{\partial}
\newcommand{\isom}{\cong}
\newcommand{\homeo}{\mathrel{\cong}} 
\newcommand{\group}[2]{{\langle #1 \st #2 \rangle}}
\newcommand{\cover}[1]{\widetilde{#1}}

\newcommand{\NP}{\textsc{NP}}
\newcommand{\FNP}{\textsc{FNP}}

\newcommand{\interior}{{\mathrm{int}\:}}
\newcommand{\closure}{\operatorname{closure}}

\newcommand{\cut}{\backslash\backslash}
\newcommand{\subgp}[1]{{\langle #1 \rangle}}
\newcommand{\twist}{\mathrel{%
    \stackrel{\sim}{\smash{\times}\rule{0pt}{0.6ex}}
    }} 
\newcommand{\st}{\mathbin{\mid}} 
\newcommand{\connect}{\#} 

\begin{document}

\title{Recognising elliptic manifolds}

\author[Lackenby]{Marc Lackenby}
\address{\hskip-\parindent
  Mathematical Institute\\
  University of Oxford\\
  Oxford OX2 6GG, United Kingdom}
\email{lackenby@maths.ox.ac.uk}

\author[Schleimer]{Saul Schleimer}
\address{\hskip-\parindent
  Mathematics Institute\\
  University of Warwick\\
  Coventry CV4 7AL, United Kingdom}
\email{s.schleimer@warwick.ac.uk}

\thanks{This work is in the public domain.}

\date{\today}

\begin{abstract}
We show that the problem of deciding whether a closed three-manifold admits an elliptic structure lies in \NP.
Furthermore, determining the homeomorphism type of an elliptic manifold lies in the complexity class \FNP.
These are both consequences of the following result.
Suppose that $M$ is a lens space which is neither $\RP^3$ nor a prism manifold.
Suppose that $\calT$ is a triangulation of $M$.
Then there is a loop, in the one-skeleton of the $86^\thsup$ iterated barycentric subdivision of $\calT$, whose simplicial neighbourhood is a Heegaard solid torus for $M$.
\end{abstract}






\maketitle

\section{Introduction}
\label{Sec:Intro}

Compact orientable three-manifolds have been classified in the following sense:
there are algorithms that, given two such manifolds, determine if they are homeomorphic~\cite{Kuperberg19, ScottShort, AschenbrennerFriedlWilton}.
Kuperberg~\cite[Theorem~1.2]{Kuperberg19} has further shown that this problem is no worse than elementary recursive.
Beyond this, very little is known about the computational complexity of the homeomorphism problem.

All known solutions rely on the geometrisation theorem, due to Perelman~\cite{Perelman1, Perelman2, Perelman3}.
This motivates the following closely related problem:
given a compact orientable three-manifold, determine if it admits one of the eight Thurston geometries.
This problem has an exponential-time solution using normal surface theory and, again, geometrisation.
See~\cite[Section~8.2]{Kuperberg19} for a closely related discussion.

Here we give a much better upper bound in an important special case.
Recall that the \emph{elliptic} manifolds are those admitting spherical geometry.
The decision problem \textsc{Elliptic manifold} takes as input a triangulation $\calT$, of a compact connected three-manifold $M$, and asks if $M$ is elliptic.

\begin{theorem}
\label{Thm:Elliptic}
The problem \textsc{Elliptic manifold} lies in \NP.
\end{theorem}

If a three-manifold is elliptic, then it is reasonable to ask which elliptic manifold it is.
This is a \emph{function problem} as the desired output is more complicated than simply ``yes'' or ``no''.
The problem \textsc{Naming elliptic} takes as input a triangulation of a compact connected three-manifold, which is promised to admit an elliptic structure, and requires as output the manifold's Seifert data.
Some elliptic three-manifolds admit more than one Seifert fibration;
in this case, the output is permitted to be the data for any of these.

\begin{theorem}
\label{Thm:NamingElliptic}
The problem \textsc{Naming elliptic} lies in \FNP.
\end{theorem}

One precursor to \refthm{Elliptic} is that \textsc{Three-sphere recognition} lies in \NP.
The first proof of this~\cite[Theorem~15.1]{Schleimer11} uses Casson's technique of \emph{crushing} normal two-spheres as well as Rubinstein's \emph{sweep-outs}, derived from almost normal two-spheres.
There is another proof, due to Ivanov~\cite[Theorem~2]{Ivanov08}, that again uses crushing but avoids the machinery of sweep-outs.
Ivanov also shows that the problem of recognising the solid torus lies in \NP.

Our results rely on this prior work in a crucial but non-obvious fashion.
By geometrisation, a three-manifold $M$ is elliptic if and only if $M$ is finitely covered by the three-sphere.
Thus, one might hope to prove Theorems~\ref{Thm:Elliptic} and~\ref{Thm:NamingElliptic} by exhibiting such a finite cover together with a certificate that the cover is $S^3$.
However, consider the following examples.
Let $F_n$ denote the $n^\thsup$ Fibonacci number.
There is a triangulation of the lens space $L(F_n, F_{n-1})$ with $n$ tetrahedra.
The degree of the universal covering is $F_n$;
since this grows exponentially in $n$ it cannot be used directly in an \NP~certificate.
See~\cite[Section~2]{Kuperberg18} for many more examples of this phenomenon.

Instead we use the following:
any elliptic three-manifold has a cover, of degree at most sixty, which is a lens space.
Thus Theorems~\ref{Thm:Elliptic} and~\ref{Thm:NamingElliptic} reduce, respectively, to the problem of deciding whether a three-manifold is a lens space and, if so, naming it.

Our approach to certifying lens spaces is conceptually simple.
Suppose that $U \homeo V \homeo S^1 \cross D^2$ are solid tori.
A lens space $M$ can be obtained by gluing $U$ and $V$ along their boundaries.
This gives a \emph{Heegaard splitting} for $M$.
This decomposition of $M$ is unique up to ambient isotopy~\cite[Th\'eor\`eme~1]{BonahonOtal83}.
We call $S^1 \cross \{0\} \subset S^1 \cross D^2$ a \emph{core curve} for the solid torus.
We say that a simple closed curve $\gamma \subset M$ is a \emph{core curve} for $M$ if it is isotopic to a core curve for $U$ or for $V$.
We certify that $M$ is a lens space by exhibiting such a core curve.

This approach is inspired by the results of~\cite{Lackenby:CoreCurves}.
There Lackenby shows that, for any handle structure of a solid torus satisfying some natural conditions, there is a core curve that lies nicely with respect to the handles.
Specifically, the curve lies within the union of the zero- and one-handles and its intersection with each such handle is one of finitely many types.
This list of types is universal, in the sense that it does not depend on the handle structure.
The handle structure must satisfy some hypotheses, but these hold for any handle structure that is dual to a triangulation.
Using~\cite[Theorem~4.2]{Lackenby:CoreCurves} we give an explicit bound on the ``combinatorial length'' of the core curve.
For a triangulation $\calT$ and positive integer $n$, we let $\calT^{(n)}$ denote the triangulation obtained from $\calT$ by performing barycentric subdivision $n$ times.

\begin{theorem}
\label{Thm:DerivedSolidTorus}
Suppose that $\calT$ is a triangulation of the solid torus $M$.
Then $M$ contains a core curve that is a subcomplex of $\calT^{(51)}$.
\end{theorem}

Using this we prove the following technical result.

\begin{theorem}
\label{Thm:LensSpaceCurve}
Suppose that $M$ is a lens space other than $\RP^3$.
Suppose that $\calT$ is a triangulation of $M$.
Then there is a simple closed curve $C$ that is a subcomplex of $\calT^{(86)}$,
such that the exterior of $C$ is either a solid torus or a twisted $I$--bundle over a Klein bottle.
\end{theorem}

This is proved by placing a Heegaard torus $S$ into almost normal form in the triangulation $\calT$, using Stocking's work~\cite[Theorem~1]{Stocking}.
We then cut along this torus to get two solid tori, which inherit handle structures $\calH_1$ and $\calH_2$.
We could apply \refthm{DerivedSolidTorus} to each of these, and we would then obtain core curves of the solid tori.
However, their intersection with each tetrahedron of $\calT$ would not necessarily be of the required form.
In particular, the intersection between each of these curves and any tetrahedron of $\calT$ would not be in one of finitely many configurations; 
thus there would be no way of showing that it lay in $\calT^{(86)}$.
The reason for this is that each tetrahedron of $\calT$ may contain many handles of $\calH_1$ and $\calH_2$.
However, in this situation, all but a bounded number of these handles would lie between normally parallel triangles and squares of $S$.
Hence, they lie in the \emph{parallelity bundle} for $\calH_1$ or $\calH_2$.
(The parallelity bundle was first introduced by Kneser~\cite{Kneser28};
we follow the treatment given in~\cite{Lackenby:Composite}.)
The strategy behind \refthm{LensSpaceCurve} is to ensure that one of the core curves does not intersect these handles by using the results from~\cite{Lackenby:CoreCurves}.
It then intersects each tetrahedron of $\calT$ in one of finitely many types, and with some work, we show that it is in fact simplicial in $\calT^{(86)}$.

Using Theorems~\ref{Thm:DerivedSolidTorus} and~\ref{Thm:LensSpaceCurve} we prove our main technical result.

\begin{theorem}
\label{Thm:LensSpaceCore}
Suppose that $M$ is a lens space which is neither a prism manifold nor a copy of $\RP^3$.
Suppose that $\calT$ is a triangulation of $M$.
Then the iterated barycentric subdivision $\calT^{(86)}$ contains a core curve of $M$ in its one-skeleton.
Furthermore, $\calT^{(139)}$ contains in its one-skeleton the union of the two core curves.
\end{theorem}

\subsection{Other work}

We announced Theorems~\ref{Thm:Elliptic} and~\ref{Thm:NamingElliptic} in 2012, at Oberwolfach~\cite{LackenbySchleimer12}.
Motivated by this, Kuperberg~\cite[Theorem~1.1]{Kuperberg18} showed that
the \emph{function promise problem} \textsc{Naming lens space} has a polynomial-time solution.
That is, there is a polynomial-time algorithm that,
given a triangulated lens space $M$, determines its lens space coefficients.
His work, together with \refthm{Elliptic} and parts of \refsec{CertificateElliptic}, can be used to give another proof of \refthm{NamingElliptic}.

The work of Haraway and Hoffman~\cite{HarawayHoffman19} is also relevant here.
In particular, in \refsec{IBundles}, we rely upon their result~\cite[Theorem~3.6]{HarawayHoffman19} that the decision problems \textsc{Recognising $T^2 \times I$} and \textsc{Recognising $K^2 \twist I$} lie in \NP.

\subsection{Outline of paper}

For us, all manifolds are given via a triangulation;
we work in the PL category throughout. 

In \refsec{LensPrism}, we remind the reader of some elementary facts about lens spaces and \emph{prism manifolds}.
The lens spaces which are also prism manifolds are exceptional cases in our analysis (as can be seen in the statement of \refthm{LensSpaceCore}, for example).
In \refsec{NormalAlmostNormal}, we recall the definition of an (almost) normal surface in a triangulated three-manifold.
In \refsec{HandleStructures}, we discuss handle structures for three-manifolds.
In \refsec{SolidTori}, we give the background needed to state~\cite[Theorem~4.2]{Lackenby:CoreCurves}.
This result places the core curve of a solid torus into a controlled position with respect to a handle structure for the solid torus.
In \refsec{Barycentric}, we translate this back to triangulations.
In particular, we prove \refthm{DerivedSolidTorus}.

In \refsec{NicelyEmbedded}, we say what it means for one three-manifold to be \emph{nicely embedded} into another.
Our eventual goal is to show that one of the Heegaard solid tori of the lens space $M$ is nicely embedded within the triangulation $\calT$ of $M$.
We show that, in this situation, the core curve of this solid torus can be arranged to be simplicial in an iterated barycentric subdivision of $\calT$.

In \refsec{Parallelity}, we recall the notion of parallelity bundles in a handle structure.
In addition, we also discuss \emph{generalised} parallelity bundles, also from \cite{Lackenby:Composite}, where it was shown that such bundles may be chosen to have incompressible horizontal boundary~\cite[Proposition~5.6]{Lackenby:Composite}.
In our situation, this implies that the generalised parallelity bundle has horizontal boundary being a collection of annuli and discs lying in the Heegaard torus.
In \refsec{ProofMain}, we bring these results together and prove Theorems~\ref{Thm:LensSpaceCore} and~\ref{Thm:LensSpaceCurve}.
This section may be viewed as the heart of the paper.

In \refsec{Elliptic}, we recall the classification of elliptic three-manifolds into lens spaces, prism manifolds, and \emph{platonic} manifolds.
We show that any platonic manifold has a cover, of degree at most sixty, which is a lens space.
We also show how its Seifert data can be extracted from the cover and from the homology of the manifold.
In \refsec{CertificateElliptic}, we certify elliptic manifolds, thereby completing the proof of Theorems~\ref{Thm:Elliptic} and~\ref{Thm:NamingElliptic}.

\subsection{The role of geometrisation in this paper}
\label{Sec:Geometrisation}

Very little of this paper relies on the proof of Thurston's geometrisation conjecture by Perelman \cite{Perelman1, Perelman2, Perelman3}. 
The proofs of Theorems~\ref{Thm:DerivedSolidTorus},~\ref{Thm:LensSpaceCurve} and~\ref{Thm:LensSpaceCore} make no appeal to geometrisation. 
As a result, our proof that lens space recognition is in NP also does not rely on geometrisation.
As explained above, we certify other elliptic manifolds by exhibiting a cover that is a lens space. 
The fact that any manifold covered by a lens space is elliptic seems to rely on the resolution of the spherical space form conjecture, which was a consequence of geometrisation. 
However, the certification of prism manifolds can avoid the use of geometrisation by appealing to earlier results on the spherical space form conjecture by Livesay~\cite{Livesay60} and Myers~\cite{Myers81}.
However, the certification of platonic manifolds \emph{does} seem to rely on geometrisation; 
see the proof of \refprop{PlatonicCriterion}.

\subsection{Acknowledgements}

We thank the referee for their detailed and insightful comments. 
We also thank the organisers of the conference on Geometric Topology of Knots at the University of Pisa, where this work was initiated.
ML was partially supported by the Engineering and Physical Sciences Research Council (grant number EP/Y004256/1).

\section{Lens spaces and prism manifolds}
\label{Sec:LensPrism}


In this section, we gather a few facts about lens spaces and prism manifolds.

Fix an orientation of the two-torus $T$.
Suppose that $\lambda$ and $\mu$ are simple closed oriented curves in $T$.
We write $\lambda \cdot \mu$ for the algebraic intersection number of $\lambda$ and $\mu$.
If $\lambda \cdot \mu = 1$ then we call the ordered pair $\subgp{\lambda, \mu}$ a \emph{framing} of $T$.
In this case we may isotope $\lambda$ and $\mu$ so that $x = \lambda \cap \mu$ is a single point.
This done, $\lambda$ and $\mu$ generate $\pi_1(T, x) \isom \ZZ^2$.




Thus, for any simple closed oriented essential curve $\alpha$ in $T$ we may write $\alpha = p\lambda + q\mu$, with $p$ and $q$ coprime.
Note that $\lambda \cdot \alpha = q$ and $\alpha \cdot \mu = p$.
We say that $\alpha$ has \emph{slope} $q/p$.
Note that if $\beta$ has slope $s/r$ then $\alpha \cdot \beta = \pm (ps - qr)$.
We say that $\alpha$ and $\beta$ are \emph{Farey neighbours} if $\alpha \cdot \beta = \pm 1$.

Suppose $U = D^2 \cross S^1$ is a solid torus.
Fix $x \in \bdy D^2$ and $y \in S^1$.
The boundary $T = \bdy U$ has a framing coming from
taking $\lambda = \{x\} \cross S^1$ and $\mu = \bdy D^2 \cross \{y\}$.

\begin{definition}
\label{Def:Lens}
A three-manifold $M$, obtained by gluing a pair of solid tori $U$ and $V$ via a homeomorphism of their boundaries, is called a \emph{lens space} if it has finite fundamental group.
\end{definition}

Under this definition $S^2 \cross S^1$ is not a lens space;
we use this convention because we are here interested in elliptic manifolds.
We use $L(p, q)$ to denote the lens space obtained by attaching the meridian slope
in $\bdy V$ to the slope $q/p$ in $\bdy U$.
We call $p$ and $q$ \emph{coefficients} for the lens space.
We now record several facts.

\begin{fact}
\label{Fac:Fund}
$\pi_1(L(p, q)) \isom \ZZ/p\ZZ$.
\end{fact}

Note for any $q$ we have $L(1, q) \homeo S^3$.

\begin{fact}
\label{Fac:Homeo}
\cite{Reidemeister35}
$L(p', q')$ is homeomorphic to $L(p, q)$ if and only if $|p'| = |p|$ and $q' = \pm q^{\pm 1} \mod p$.
\end{fact}

\begin{fact}
\label{Fac:Cover}
The double cover of $L(2p, q)$ is $L(p, q)$.
\end{fact}

Notice $L(p, q)$ double covers both $L(2p, q)$ and $L(2p, p-q)$.
For example, $L(8, 1)$ and $L(8, 3)$ are both covered by $L(4, 1)$.
However, $L(8, 1)$ and $L(8, 3)$ are not homeomorphic according to \reffac{Homeo}.
Thus one cannot recover the coefficients of a lens space just by knowing a double cover.

\begin{fact}
\label{Fac:Glue}
Suppose that $\alpha$ and $\beta$ are Farey neighbours in $T$, with slopes $q/p$ and $s/r$.
Suppose that $\gamma$ has slope $q'/p'$ in $T$.
The three-manifold $M = U \cup (T \cross I) \cup V$, formed by attaching the meridian of $U$ along $\alpha \cross 0$ and attaching the meridian of $V$ along $\gamma \cross 1$, is homeomorphic to the lens space $L(-p'q + q'p, p's - q'r)$.
\end{fact}

We use $\twist$ to denote a twisted product.
We write $K = K^2 = S^1 \twist S^1$ for the Klein bottle.
We write $K \twist I$ for the orientation $I$--bundle over the Klein bottle.

Recall that $K$ contains exactly four essential simple closed curves, up to isotopy.
These are the cores $\alpha$ and $\alpha'$ of the two M\"obius bands, their common boundary $\delta$, and the fibre $\beta$ of the bundle structure $K = S^1 \twist S^1$.
Thus $\pi_1(K) \isom \pi_1(K \twist I)$ has a presentation
\[
\group{a,b}{aba^{-1} = b^{-1}}
\]
where $a = [\alpha]$ and $b = [\beta]$.
This presentation is not canonical, as we could have chosen $\alpha'$ instead of $\alpha$.
Let $\rho \from T \to K$ be the orientation double cover.
Thus we have
\[
\rho_*(\pi_1(T)) \isom \subgp{a^2, b} < \pi_1(K).
\]
Since $a^2 = [\delta]$ this generating set for $\pi_1(T)$ gives a canonical framing of $T$ (up to the choice of orientations).
The identification of $T$ and $\bdy (K \twist I)$ now gives us a canonical framing of the latter.

\begin{definition}
\label{Def:Prism}
A three-manifold $M$, obtained by gluing an $I$--bundle $K \twist I$ to a solid torus $W$ via a homeomorphism of their boundaries, is called a \emph{prism manifold} if it has finite fundamental group.
\end{definition}

We use the notation $P(p, q)$ to denote the three-manifold obtained by gluing the meridian slope of a solid torus $W$ to the slope $a^{2p}b^q$ in $\bdy (K \twist I)$.
We call $p$ and $q$ \emph{coefficients} of $P(p, q)$.

\begin{fact}
\label{Fac:FundPrism}
$\pi_1(P(p, q)) \isom \group{a, b}{aba^{-1} = b^{-1}, a^{2p} = b^{-q}}$.
\end{fact}

For example, we have $P(1, 1) \homeo L(4, 3)$;
see \reflem{PrismIsLens}.
On the other hand we have $P(1, 0) \homeo \RP^3 \connect \RP^3$
and $P(0, 1) \homeo S^2 \cross S^1$.
Since the fundamental groups are infinite, we do not admit $P(1, 0)$ or $P(0, 1)$ as prism manifolds.

\begin{lemma}
\label{Lem:EmbeddedKleinBottle}
Let $M$ be an irreducible oriented closed non-Haken three-manifold.
Then $M$ contains an embedded Klein bottle if and only if it is a prism manifold.
In particular, a lens space contains an embedded Klein bottle if and only if it is a prism manifold.
\end{lemma}

\begin{proof}
By construction, prism manifolds contain an embedded Klein bottle. We need to establish the
converse. Suppose that $M$ contains an embedded Klein bottle $K$. By the orientability of $M$,
the regular neighbourhood $N(K)$ is the orientable $I$--bundle over the Klein bottle.
Let $T = \bdy N(K)$ be the boundary torus.
Since $M$ is non-Haken, $T$ is compressible along a disc $D$.
Note that $D$ has interior disjoint from $N(K)$.
Compressing $T$ along $D$ gives a sphere, which bounds a ball $B$ by the irreducibility of $M$.
Note that $B$ is disjoint from $K$.
Thus $T$ bounds a solid torus with interior disjoint from $N(K)$.
Therefore, $M$ is a prism manifold.
\end{proof}

We now give a proof of~\cite[Corollary~6.4]{BredonWood69}.
For another account, with further references, see~\cite[Theorem~1.2]{GeigesThies23}.
This proof is included for the convenience for the reader.

\begin{lemma}
\label{Lem:PrismIsLens}
$P(p, q)$ is a lens space if and only if $q = 1$ and $p \neq 0$.
In this case, $P(p, 1) \homeo L(4p, 2p + 1)$.
\end{lemma}


\begin{proof}
Recall $\gcd(p, q) = 1$ for any slope $q/p$.
So if $q = 0$ then $p = 1$;
likewise, if $p = 0$, then $q = 1$.
In these cases, as noted above, $P(1, 0)$ and $P(0, 1)$ are not lens spaces.

Suppose that $q \geq 2$.
Taking $P = P(p, q)$ we note that $\pi_1(P)$ has as a quotient group
\[
D_{2q} = \group{a,b}{aba^{-1} = b^{-1}, a^2 = b^q = 1}
\]
the dihedral group of order $2q$.
This is not cyclic, so $P$ is not a lens space.

Thus we may assume that $q = 1$ and $p \neq 0$.
Set $P = P(p, 1)$.
We note that
\[
\pi_1(P) = \group{a,b}{aba^{-1} = b^{-1}, a^{2p} = b^{-1}}
         \isom \group{a}{a^{4p} = 1}.
\]
This implies that $P(p,1)$ is the unique manifold, up to homeomorphism, that is a lens
space and prism manifold with fundamental group of order $4p$.
We now give a direct proof that $P = P(p,1)$ is a lens space.
Recall that if we glue a pair of solid tori along an annulus, primitive in at least one of them, the result is a solid torus.
Write $K \twist I = Q \cup R$ as a union of solid tori, each the orientation $I$--bundle over a M\"obius band.
Note $Q \cap R$ is a vertical annulus in $K \twist I$.
Set $A = Q \cap \bdy (K \twist I)$.
Note that $A$ is not primitive in $Q$, as it crosses the meridian disc of $Q$ twice.

Recall that $P = (Q \cup R) \cup W$, all solid tori.
When $q = 1$ we find that the annulus $A$ is primitive in $W$.
Thus $Q \cup W$ is a solid torus and so $P = (Q \cup W) \cup R$ is a lens space.
This proves the first half of the lemma.

To finish we must identify the lens space coefficients of $P = P(p, 1)$.
Let $\gamma$ be the $1/2$ slope on the boundary of a solid torus $V$.
Since $\gamma$ crosses the meridian exactly twice, it bounds a M\"obius band $M$ in $V$.
If we double $V$ across its boundary, we obtain $S^2 \cross S^1$.
Note $M$ doubles to give a Klein bottle.
We now alter the double by opening it along an annulus neighbourhood of $\gamma$ inside of $\bdy V$ and regluing with a twist to obtain the lens space $L$.
Note the Klein bottle persists, so $L$ is again a prism manifold.
The image of the meridian of $V$ under the $p$--fold twist about $\gamma$ is the gluing slope, $(2p + 1)/4p$.

Thus $L(4p, 2p + 1)$ is a lens space, is a prism manifold, and has fundamental group of order $4p$.
Thus $L(4p, 2p + 1) \homeo P(p, 1)$ and we are done.
\end{proof}

\begin{lemma}
\label{Lem:PrismCover}
Suppose that $q/p$ is a slope with $s/r$ as a Farey neighbour.
Then $P(p, q)$ is double covered by the lens space $L \homeo L(2pq, ps + qr)$.
The subgroup of $\pi_1(L)$ that is fixed by the deck group, has order $2p$.
\end{lemma}

\begin{proof}
Recall that $T \cross I$ is a double cover of $K \twist I$.
Let $U$ and $V$ be solid tori, whose disjoint union double covers the solid torus $W$.
Thus the lens space $L = U \cup (T \cross I) \cup V$ double covers $P = W \cup (K \twist I)$.
Let $\alpha$ and $\beta$ be simple closed curves in $T$, lifting $a^2$ and $b$ in $K$.
Thus $\subgp{\alpha,\beta}$ gives a framing of $T$.
The meridians of $U$ and $V$ have slopes $q/p$ and $-q/p$ with respect to this framing.
\reffac{Glue} now implies $L \homeo L(2pq, ps+qr)$.

The above decomposition of $L$ gives a presentation of $\pi_1(L)$.  Abelianising, we obtain the following.
\[
\pi_1(L) \isom \group{\alpha, \beta}{p\alpha + q\beta = 0, p\alpha - q\beta = 0}
\]
The elements correspond to the integer lattice points (up to translation)
in the parallelogram with vertices at $(0, 0)$, $(p, -q)$, $(2p, 0)$, and $(p, q)$.
The deck group fixes $\alpha$ while sending $\beta$ to $\beta^{-1}$.
The fixed points under this action are the lattice points with second coordinate zero.
So the fixed subgroup has order $2p$.
\end{proof}




We use \reflem{PrismCover} to decide if a given manifold is a prism manifold $P(p, q)$,
assuming that a double cover of the manifold is a lens space.

\section{Normal and almost normal surfaces}
\label{Sec:NormalAlmostNormal}

\begin{definition}
\label{Def:NormalArc}
Suppose that $f$ is a two-simplex.
An arc, properly embedded in $f$, is \emph{normal} if it misses the vertices of $f$ and has endpoints on distinct edges.
\end{definition}

\begin{definition}
\label{Def:TriangleSquare}
Suppose that $\Delta$ is a tetrahedron.
A disc $D$, properly embedded in $\Delta$ and transverse to the edges of $\Delta$, is
\begin{itemize}
\item a \emph{triangle} if $\bdy D$ consists of three normal arcs;
\item a \emph{square} if $\bdy D$ consists of four normal arcs.
\end{itemize}
In either case $D$ is a \emph{normal disc}.
\end{definition}

\begin{definition}
\label{Def:NormalSurface}
Suppose that $(M, \calT)$ is a triangulated three-manifold.
A surface $S$, properly embedded in $M$, is \emph{normal} if, for each tetrahedron $\Delta \in \calT$, the intersection $S \cap \Delta$ is a disjoint collection of triangles and squares.
\end{definition}

\begin{definition}
Suppose that $\Delta$ is a tetrahedron.
A surface $E$, properly embedded in $\Delta$ and transverse to the edges of $\Delta$, is
\begin{itemize}
\item an \emph{octagon} if $E$ is a disc with $\bdy E$ consisting of eight normal arcs;
\item a \emph{tubed piece} if $E$ is an annulus obtained from two disjoint normal discs by attaching a tube that runs parallel to an arc of an edge of $\Delta$.
\end{itemize}
In either case $E$ is an \emph{almost normal piece}.
\end{definition}

\begin{figure}
  \includegraphics{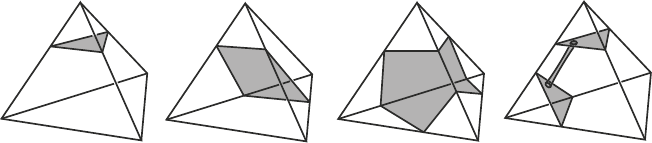}
  \caption{Left to right: triangle, square, octagon, tubed piece.}
  \label{Fig:NormalAlmostNormal}
\end{figure}

\begin{definition}
\label{Def:AlmostNormalSurface}
Suppose that $(M, \calT)$ is a triangulated three-manifold.
A surface $S$, properly embedded in $M$, is \emph{almost normal} if, for each tetrahedron $\Delta \in \calT$, the intersection $S \cap \Delta$ is a disjoint collection of triangles and squares except for precisely one tetrahedron, where the collection additionally contains exactly one almost normal piece.
\end{definition}

\begin{definition}
A Heegaard surface $S$ for a three-manifold $M$ is \emph{strongly irreducible} if it does not have disjoint compressing discs emanating from opposite sides of $S$.
\end{definition}

We note that the genus one Heegaard surface for a lens space is strongly irreducible.
We now have a result due to Stocking~\cite[Theorem~1]{Stocking}, following work of Rubinstein.

\begin{theorem}
\label{Thm:AlmostNormalHeegaard}
Suppose that $M$ is a closed, connected, oriented three-manifold equipped with a triangulation $\calT$.
Suppose that $H$ is a strongly irreducible Heegaard surface for $M$.
Then $H$ is (ambiently) isotopic to a surface which is almost normal with respect to $\calT$.  \qed
\end{theorem}

Although almost normal surfaces are useful, they can sometimes be technically challenging to work with.
We therefore apply the following result.

\begin{proposition}
\label{Prop:FromAlmostNormalToNormal}
Suppose that $S$ is almost normal with respect to a triangulation $\calT$. Then $S$ is isotopic to a surface which is normal with respect to the first barycentric subdivision $\calT^{(1)}$.
\end{proposition}


To prove this, we first require some lemmas.

\begin{lemma}
\label{Lem:NormaliseSubsurfaceOfBoundary}
Suppose that $\calT$ is a triangulation of a three-ball $B$.
Suppose that for each tetrahedron $\Delta$ of $\calT$ the preimage of $\bdy B$ in $\Delta$ is either empty or is a single vertex, edge, or face.
Suppose that $F$ is a subsurface of $\bdy B$ so that $\bdy F$ is normal with respect to $\bdy B$ and intersects each edge of $\calT$ at most once.
Then, after pushing the interior of $F$ slightly into the interior of $B$, the resulting properly embedded surface $F'$ is normal.
\end{lemma}


\begin{proof}
Consider any tetrahedron $\Delta$ of $\calT$.
We divide into cases depending on the number of vertices of $\Delta$ contained in $F$.

If there are none, then $F' \cap \Delta$ is empty.
Suppose that there is exactly one vertex $v$ of $\Delta$ contained in $F$.
Then $F' \cap \Delta$ is a normal triangle separating $v$ from the remaining three vertices of $\Delta$.
This triangle meets $\bdy B$ in the set $(\bdy F) \cap \Delta$.
Suppose instead that there are exactly two vertices of $\Delta$ contained in $F$.
Then there is an edge $e$ of $\Delta$ contained in $F$.
In this case $F' \cap \Delta$ is a normal square separating $e$ from the remaining two vertices of $\Delta$.
Again, this square meets $\bdy B$ in the set $(\bdy F) \cap \Delta$.
Suppose instead that there are exactly three vertices of $\Delta$ contained in $F$.
Then there is a face $f$ of $\Delta$ contained in $F$.
In this case $F' \cap \Delta$ is again a normal triangle, separating $f$ from the final vertex of $\Delta$.
Also, this triangle is disjoint from $\bdy B$.
Note that, by assumption, we cannot have all four vertices of $\Delta$ contained in $F$.
\end{proof}

\begin{lemma}
\label{Lem:NormaliseMultipleSubsurfacesOfBoundary}
Suppose that $\calT$ is a triangulation of a three-ball $B$.
Suppose that for each tetrahedron $\Delta$ of $\calT$ the preimage of $\bdy B$ in $\Delta$ is either empty or is a single vertex, edge, or face.
Let $F_1, \dots, F_n$ be a collection of subsurfaces of $\bdy B$.
Suppose that each $\bdy F_i$ is normal and intersects each edge of $\calT$ at most once.
Suppose also that for each $i \not= j$, $\bdy F_i$ and $\bdy F_j$ are disjoint and either $F_i \subset F_j$ or $F_j \subset F_i$.
Then we may push the interiors of $F_1, \dots, F_n$ slightly into the interior of $B$ so that $F'$, the resulting union of surfaces, is properly embedded and normal.
\end{lemma}

\begin{proof}
We apply the construction in the proof of \reflem{NormaliseSubsurfaceOfBoundary} to each $F_i$, to form a normal surface $F'_i$.
We note that when $i \not=j$, we can arrange for $F'_i$ and $F'_j$ to be disjoint, as follows.
There is some $F_i$ with the property that no other $F_j$ lies inside it.
For this $F_i$, form the resulting surface $F'_i$, which we can view as lying extremely close to $\bdy B$.
When we form the remaining surfaces $F'_j$, we can ensure that they are disjoint from this $F'_i$.
In this way, the required surface is constructed recursively.
\end{proof}

\begin{proof}
[Proof of \refprop{FromAlmostNormalToNormal}]
We first specify the intersection between $S$ and the edges of $\calT^{(1)}$ lying in the one-skeleton of $\calT$.
By assumption, $S$ has an almost normal piece $P$ in some tetrahedron $\Delta$ of $\calT$.
The intersection between $P$ and each one-simplex of $\Delta$ is at most two points.
If it is exactly two points, then we arrange for these to lie in distinct edges of $\calT^{(1)}$.
We then arrange that the remaining points of intersection between $S$ and the one-skeleton of $\calT$ are disjoint from the vertices of $\calT^{(1)}$.

The intersection between $S$ and each two-simplex $F$ of $\calT$ consists of a collection of arcs that are normal with respect to $\calT$.
We may realise these arcs $F \cap S$ as a concatenation of normal arcs in the two-skeleton of $\calT^{(1)}$, in such a way that each arc of $F \cap S$ intersects each edge of $\calT^{(1)}$ at most once.

Thus, we have specified the intersection between $S$ and the simplices of $\calT^{(1)}$ lying in the two-skeleton of $\calT$.
The remainder of the three-manifold is a collection of tetrahedra of $\calT$. Consider any such tetrahedron $\Delta$.
This inherits a triangulation from $\calT^{(1)}$.
We have already specified $S \cap \bdy \Delta$.
This is a collection of normal curves in $\bdy \Delta$.

Suppose first that $\Delta$ does not contain the almost normal piece of $S$. Then $S \cap \Delta$ is a collection of triangles and squares in $\Delta$. Their boundary is a collection $C_1, \dots, C_n$ of normal curves in $\bdy \Delta$.
We need to specify, for each $C_i$, a subsurface $F_i$ of $\bdy \Delta$ that it bounds.
For each $C_i$ bounding a triangle in $\Delta$, we pick $F_i$ so that it contains a single vertex of $\Delta$.
The remaining $C_i$ bound normal squares in $\Delta$.
We pick the $F_i$ that these curves $C_i$ bound so that they are nested.
Then applying \reflem{NormaliseMultipleSubsurfacesOfBoundary}, we realise the discs bounded by these curves in $\Delta$ as normal surfaces with respect to $\calT^{(1)}$.

Suppose now that $\Delta$ does contain the almost normal piece $P$ of $S$.
Then $S \cap \Delta$ is equal to $P$ plus possibly some triangles and squares.
In the case where $P$ is an octagon, its boundary divides $\bdy \Delta$ into two discs, and we pick one of these discs to be the relevant $F_i$.
In the case where $P$ is obtained by tubing together two normal discs in $\Delta$, the boundary curves of these discs cobound an annulus in $\bdy \Delta$, and we set $F_i$ to be this annulus.
The remaining components of $S \cap \Delta$ are triangles and squares, and squares can only arise in the case where $P$ is tubed.
For each triangle with boundary $C_j$, we pick $F_j \subset \bdy \Delta$ so that it contains a single vertex of $\Delta$.
For each square with boundary $C_j$, we pick $F_j \subset \bdy \Delta$ so that
it is disjoint from the annulus $F_i$ considered above.
Applying \reflem{NormaliseMultipleSubsurfacesOfBoundary} again, we arrange for $S \cap \Delta$ to be normal with respect to $\calT^{(1)}$.
\end{proof}

\section{Handle structures}
\label{Sec:HandleStructures}

Suppose that $M$ is a compact, connected, oriented three-manifold.
Suppose that $\calH$ is a handle decomposition of $M$.
For example, we may obtain $\calH$ by taking the dual of a triangulation.

\begin{definition}
\label{Def:Standard}
Suppose that $S$ is a surface, properly embedded in $M$.
We say that $S$ is \emph{standard} with respect to $\calH$ if the following properties hold.
\begin{itemize}
\item
The intersection of $S$ and any zero-handle $D^0 \cross D^3$ is a disjoint union of properly embedded discs.
\item
The intersection of $S$ and any one-handle $D^1 \cross D^2$ is of the form $D^1 \cross A$, where $A$ is a disjoint union of arcs properly embedded in $D^2$.
\item
The intersection of $S$ and any two-handle $D^2 \cross D^1$ is of the form $D^2 \cross P$, where $P$ is a finite collection of points in the interior of $D^1$.
\item
The intersection of $S$ and any three-handle $D^3 \cross D^0$ is empty. \qedhere
\end{itemize}
\end{definition}

\begin{definition}
Let $M$ be a closed three-manifold with a handle structure $\calH$. A disc properly embedded in a zero-handle $H_0$ of $\calH$ is \emph{normal} if
\begin{itemize}
\item its boundary lies within the union of the one-handles and two-handles;
\item its intersection with each one-handle is a collection of arcs;
\item it runs over each component of intersection between $H_0$ and the two-handles in at most one arc, and this arc respects the product structure on the two-handles.
\end{itemize}
A disc properly embedded in a one-handle $H_1$ of $\calH$ is \emph{normal} if
\begin{itemize}
\item it respects the product structure on $H_1$;
\item its boundary lies within the union of the zero-handles and two-handles;
\item it runs over each component of intersection between $H_1$ and the two-handles in at most one arc. \qedhere
\end{itemize}
\end{definition}

\begin{definition}
Let $M$ be a closed three-manifold with a handle structure $\calH$. A surface properly embedded within $M$ is \emph{normal} if it is standard and its intersection with each zero-handle is normal. This implies that its intersection with each one-handle is also normal.
\end{definition}

When $S$ is a normal surface embedded in a closed triangulated three-manifold $M$, it naturally becomes a standard surface in the dual handle structure.
When a three-manifold with a handle structure is decomposed along a standard surface $S$, the resulting three-manifold $M \cut S$ inherits a handle structure.
Thus, we deduce that when a closed triangulated three-manifold is cut along a normal surface, then $M \cut S$ inherits a handle structure.

Our handle structures satisfy the following condition.

\begin{definition}
A handle structure of a three-manifold is \emph{locally small} if the following conditions hold:
\begin{itemize}
\item
the intersection between any zero-handle and the union of the one-handles consists of at most $4$ discs and
\item
the intersection between any one-handle and the union of the two-handles consists of at most $3$ discs. \qedhere
\end{itemize}
\end{definition}

The handle structure that is dual to a triangulation is locally small. Moreover, when a three-manifold $M$ with a locally small handle structure
is cut along a normal surface $S$,
the resulting handle structure is again locally small.

\section{Core curves of solid tori}
\label{Sec:SolidTori}

Over the next two sections, we prove \refthm{DerivedSolidTorus}. 
This is achieved using affine handle structures, first introduced in \cite{Lackenby:CoreCurves}.
Here is the definition.

\begin{definition}
\label{Def:AffineHandle}
An \emph{affine handle structure} on a three-manifold $M$ is a handle structure where each
zero-handle and one-handle is identified with a compact polyhedron in $\RR^3$,
so that
\begin{itemize}
\item each face of each polyhedron is convex (but the polyhedron identified with a zero-handle
need not be convex);
\item whenever a zero-handle and one-handle intersect, each component of intersection is
identified with a convex polygon in $\RR^2$, in such a way that the inclusion
of this intersection into each handle is an affine map onto a face of the
relevant polyhedron;
\item for each zero-handle $H_0$, each component of intersection with a two-handle, three-handle or $\bdy M$ is a union of faces of the polyhedron associated with $H_0$;
\item
the polyhedral structure on each one-handle is the product of a convex two-dimensional polygon and an interval. \qedhere
\end{itemize}
\end{definition}

\begin{definition}
\label{Def:CanonicalAffine}
Let $\calT$ be a triangulation of a compact three-manifold $M$ and let $\calH$ be the dual handle structure.
Then the \emph{canonical affine structure} on $\calH$ realises each zero-handle as a truncated octahedron.
Specifically, one realises each tetrahedron of $\calT$ as regular and euclidean with side length $1$, and then one slices off the edges and the vertices to form a truncated octahedron.
Each one-handle of $\calH$ corresponds to a face of $\calT$ that does not lie wholly in $\bdy M$, and hence corresponds to a pair of hexagonal faces of the truncated octahedra that are identified.
Thus, the one-handle is identified with the product of a hexagon and an interval.
(See \reffig{TruncatedOctahedron}.)
\end{definition}

\begin{figure}
  \includegraphics[width=2.5in]{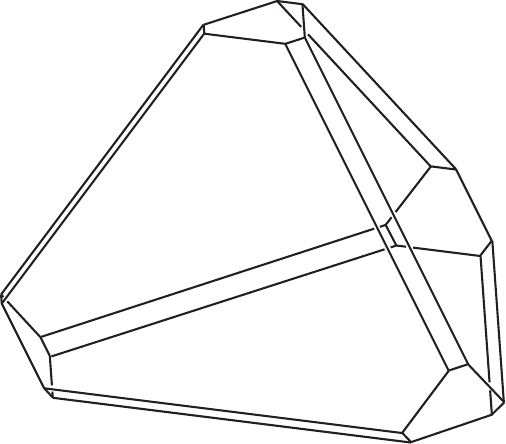}
  \caption{Each zero-handle is realised as a truncated octahedron.
    This is obtained from a tetrahedron by slicing off its vertices and edges.}
  \label{Fig:TruncatedOctahedron}
\end{figure}

The following theorem~\cite[Theorem~4.2]{Lackenby:CoreCurves} is the key result that goes into the proof of \refthm{DerivedSolidTorus}.

\begin{theorem}
\label{Thm:CoreSolidTorusAffine}
Let $\calH$ be a locally small, affine handle structure of the solid torus $M$.
Then $M$ has a core curve that intersects only the zero-handles and one-handles, that respects the product structure on the one-handles, that intersects each one-handle in at most $24$ straight arcs, and that intersects each zero-handle in at most $48$ arcs.
Moreover, this collection of arcs in each zero-handle is parallel (as a collection) to a collection of arcs $A$ in the boundary of the corresponding polyhedron. Finally, each component of $A$ intersects each face of the polyhedron
in at most $6$ straight arcs. \qed
\end{theorem}

\section{From affine handle structures to barycentric subdivisions}
\label{Sec:Barycentric}

In this section, the goal is prove Theorem \ref{Thm:DerivedSolidTorus}, which states that a core curve in a triangulated solid torus can be realised as a subcomplex in the triangulation's $51^{\mathrm{st}}$ barycentric subdivision. Many of the technicalities of this section could have been avoided if we had aimed just for a subcomplex of the $k^{\mathrm{th}}$ barycentric subdivision, for some universal but unspecified $k$ independent of the triangulation.

\begin{definition}
A triangulation of a subset $X$ of euclidean space is \emph{straight} if the inclusion of each simplex into $X$ is an affine map.
\end{definition}

\begin{definition}
\label{Def:ArcType}
Two arcs properly embedded in a polygon and disjoint from its vertices are of the same \emph{type} if there is an ambient isotopy taking one to the other, and
which keeps the arcs disjoint from the vertices.
\end{definition}

Note that a polygon with $k$ sides can support at most $2k-3$ disjoint straight arcs that are of distinct types.

\begin{lemma}
\label{Lem:ArcBecomeSimplicial}
Let $D$ be a euclidean polygon with a straight triangulation $\calT$. Let $\alpha$ be a properly embedded straight arc in $D$.
Then there is a realisation of the barycentric subdivision $\calT^{(1)}$ as a straight triangulation of $D$ that contains $\alpha$ as a
subcomplex.
\end{lemma}

\begin{proof}
Since $\alpha$ is straight and $\calT$ is straight, the intersection between $\alpha$ and the interior of each one-simplex of $\calT$ is either all the interior of the one-simplex or at most one point.
If $\alpha$ does intersect the interior of a one-simplex in a point, place a vertex of $\calT^{(1)}$ at the point of intersection.
Similarly, if $\alpha$ intersects the interior of a two-simplex, it does so in a single arc, and we place a vertex
of $\calT^{(1)}$ in the interior of this arc. Hence, $\alpha$ becomes simplicial in $\calT^{(1)}$.
(See \reffig{SimplicialCurveDisc}.)
\end{proof}

\begin{figure}
  \includegraphics[width=3.5in]{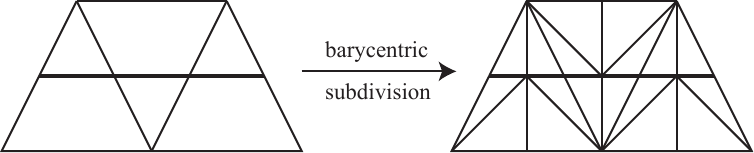}
  \caption{Making an arc simplicial}
  \label{Fig:SimplicialCurveDisc}
\end{figure}

Induction then gives the following.

\begin{lemma}
\label{Lem:ArcsBecomeSimplicial}
Let $D$ be a euclidean polygon with a straight triangulation $\calT$.
Let $A$ be a union of $k$ disjoint properly embedded straight arcs in $D$.
Then there is a realisation of $\calT^{(k)}$ as a straight triangulation of $D$ that contains $A$ as a subcomplex. \qed
\end{lemma}

However, we can improve this in certain circumstances, as follows.

\begin{lemma}
\label{Lem:TwoArcTypes}
Let $D$ be a euclidean polygon  with a straight triangulation $\calT$.
Let $A$ be a union of at most $2^k$ disjoint properly embedded straight arcs in $D$, each with endpoints that are disjoint from the vertices of $D$, and that form at most two arc types.
Then there is  a realisation of $\calT^{(k+1)}$ as a straight triangulation of $D$ that contains $A$ as a subcomplex.
\end{lemma}

\begin{proof}
We prove this by induction on $k$.
The case $k = 0$ is the statement of \reflem{ArcBecomeSimplicial}.
Let us prove the inductive step.
Since the arcs fall into at most two types, there is one component $\alpha$ of $A$ such that at most $|A|/2 \leq 2^{k-1}$ arcs lie on each side of it.
By \reflem{ArcBecomeSimplicial}, there is a realisation of $\calT^{(1)}$ as a straight triangulation of $D$ that contains $\alpha$ as a subcomplex.
Cutting $D$ along $\alpha$ gives two euclidean polygons, each of which contains at most two arc types. Inductively, the intersection between $A$ and these polygons may be made simplicial in $\calT^{(k+1)}$.
\end{proof}

\begin{lemma}
\label{Lem:SeveralArcTypes}
Let $D$ be a euclidean polygon  with a straight triangulation $\calT$. Let $A$ be a union of at most $2^k$ disjoint properly embedded arcs in $D$, each with endpoints that are disjoint from the vertices of $D$, and that form at most $n \geq 2$ arc types. Then there is  a realisation of $\calT^{(n+k-1)}$ as a straight triangulation of $D$ that contains $A$ as a
subcomplex.
\end{lemma}

\begin{proof} We prove this by induction on $n$. The induction starts with $n=2$, which is the content of \reflem{TwoArcTypes}. To prove the inductive step, suppose that $n > 2$. Pick an arc type of $A$ and let $\alpha$ be a component of $A$ of this type that is closest to other types of arcs. Make $\alpha$ simplicial in $\calT^{(1)}$ using \reflem{ArcBecomeSimplicial}. Then cut along it, to give two euclidean polygons. In each, the arcs come in at most $n-1$ types. So, by induction, a further $n+k-2$ barycentric subdivisions suffice to make $A$ simplicial.
\end{proof}

\begin{lemma}
\label{Lem:PointsBecomeVertices}
Let $D$ be a euclidean polygon with a straight triangulation $\calT$. Let $P$ be a set of at most $2^k$ points in $D$. Then there is realisation of $\calT^{(k+1)}$ as a straight triangulation of $D$ and that contains $P$ in its vertex set.
\end{lemma}

\begin{proof}
Pick a point $x$ in $D$ that is disjoint from $P$ and that also does not lie on any line containing at least two point from $P$.
We can pick a straight arc $\alpha$ through $x$, so that at most half the points of $P$ lie on each side of $\alpha$.
This can be done as follows.
Pick any straight arc $\alpha$ through $x$ that misses $P$.
If this has $|P|/2$ vertices on each side, then we have our desired arc.
If not, then pick a transverse orientation on $\alpha$ that points to the side with more than $|P|/2$ points.
Start to rotate $\alpha$ around $x$.
By our general position hypothesis on $x$, at any given moment in time, the number of points on each side of $\alpha$ can jump by at most $1$.
By the time that $\alpha$ has rotated through angle $\pi$, the number of points on the side into which it points is less than $|P|/2$.
So, at some stage, the arc contains a point of $P$ and has fewer than $|P|/2$ points of $P$ on either side of it.
At this stage, we have our required arc $\alpha$.
By \reflem{ArcBecomeSimplicial}, we can make $\alpha$ simplicial in $\calT^{(1)}$, and if $\alpha \cap P$ is non-empty, we can also make it a vertex.
Cut along $\alpha$ and apply induction.
\end{proof}

\begin{lemma}
\label{Lem:ImproperArcsSimplicial}
Let $D$ be a euclidean polygon with a straight triangulation $\calT$.
Let $A$ be a union of $k$ disjoint straight arcs, with each endpoint being a vertex of $\calT$ or a point on $\bdy D$, and with interior in the interior of $D$.
Then there is realisation of $\calT^{(k)}$ as a straight triangulation of $D$ and that contains $A$ as a subcomplex.
\end{lemma}

\begin{proof}
We prove this by induction on $k$.
Pick an arc $\alpha$ of $A$.
Since $\alpha$ is straight, the intersection between $\alpha$ and the interior of each one-simplex of $\calT$ is either all the interior of the one-simplex or at most one point.
If it does intersect the interior of this simplex in a point, place a vertex of $\calT^{(1)}$ at the point of intersection.
Similarly, if $\alpha$ intersects the interior of a two-simplex, it does so in a single arc, and we place a vertex of $\calT^{(1)}$ in the interior of this arc.
Hence, $\alpha$ becomes simplicial in $\calT^{(1)}$. We then inductively deal with the remaining $k-1$ arcs.
\end{proof}

\begin{lemma}
\label{Lem:SimplicialDiscsInPolyhedronAllParallel}
Let $\calT$ be a straight triangulation of a euclidean polyhedron $P$.
Let $C$ be a union of at most $2^k$ disjoint simple closed curves in $\bdy P$ that are simplicial in $\calT$ and that are topologically parallel in $\bdy P$.
Then there is realisation of $\calT^{(2k)}$ as a straight triangulation of $P$ such that $C$ bounds a union of disjoint properly embedded discs in $P$ that are simplicial in the triangulation.
\end{lemma}

\begin{proof}
We prove this by induction on $k$.
Since the components of $C$ are all parallel in $\bdy P$, there is some component $C'$ of $C$ that so that each component of $\bdy P - C'$ contains at most half the components of $C$.
We may realise a regular neighbourhood $N$ of $\bdy P$ as a simplicial subset of $\calT^{(2)}$.
This is homeomorphic to $\bdy P \times [0,1]$, where $\bdy P \times \{ 0 \} = \bdy P$.
The annulus $C' \times [0,1]$ may be realised as simplicial in $\calT^{(2)}$.
The curve $C' \times \{ 1 \}$ bounds a disc in $\bdy N - \bdy P = \bdy P \times \{ 1 \}$.
The union of this disc with the annulus $C' \times [0,1]$ is one of the required discs.
If we cut $P$ along this disc, the result is two polyhedra $P_1$ and $P_2$, with straight triangulations $\calT_1$ and $\calT_2$.
The intersection between $C - C'$ and each $P_i$ consists of at most $2^{k-1}$ curves.
By induction, these curves bound simplicial discs in $\calT_i^{(2k-2)}$.
Thus, $C$ bounds simplicial discs in $\calT^{(2k)}$.
\end{proof}

\begin{lemma}\label{Lem:SimplicialDiscsInPolyhedron}
Let $\calT$ be a straight triangulation of a euclidean polyhedron $P$.
Let $C$ be a union of at most $2^k$ disjoint simple closed curves in $\bdy P$ that are simplicial in $\calT$.
Suppose that the maximal number of pairwise non-parallel components of $C$ is $n$.
Then there is realisation of $\calT^{(2k+2n-2)}$ as a straight triangulation of $P$ such that $C$ bounds a union of disjoint properly embedded discs in $P$ that are simplicial in the triangulation.
\end{lemma}

\begin{proof}
We prove this by induction on $n$.
The induction starts with $n = 1$, which is the content of \reflem{SimplicialDiscsInPolyhedronAllParallel}.
To prove the inductive step, suppose that $n \geq 2$.
Pick a curve type of $C$ and let $C'$ be a component of $C$ of this type that is closest to other types of curves.
As in the proof of \reflem{SimplicialDiscsInPolyhedronAllParallel}, we may find a properly embedded disc bounded by $C'$ that is simplicial in $\calT^{(2)}$.
Cutting $P$ along this disc gives two polyhedra, each of which inherits a triangulation. In each of these polyhedra, the maximal number of non-parallel components of $C - C'$ is at most $n-1$. Thus, by induction, after barycentrically subdividing the triangulations of these polyhedra $(2k+2n-4)$ times, we obtain simplicial discs bounded by these curves. Hence, $\calT^{(2k+2n-2)}$ contains simplicial discs bounded by $C$.
\end{proof}

\begin{lemma}
\label{Lem:PushArcsIntoInterior}
Let $\calT$ be a triangulation of a polyhedron $P$. Let $A$ be a collection of disjoint simplicial arcs in $\bdy P$. Let $A'$ be the
properly embedded arcs obtained by pushing the interior of $A$ into the interior of $P$. Then, after an ambient isotopy supported in the interior of $P$,
$A'$ can be realised as simplicial in $\calT^{(2)}$.
\end{lemma}

\begin{proof}
\reffig{PushArcsIntoInterior} gives a construction of $A'$.
A regular neighbourhood $N$ of $\bdy P$ is simplicial in $\calT^{(2)}$.
This is homeomorphic to $\bdy P \times I$.
Incident to the arcs $A$ are simplicial discs of the form $A \times I$ in $N$.
We then set $A' = \bdy(A \times I) \cut A$.
\end{proof}

\begin{figure}
  \includegraphics[width=4in]{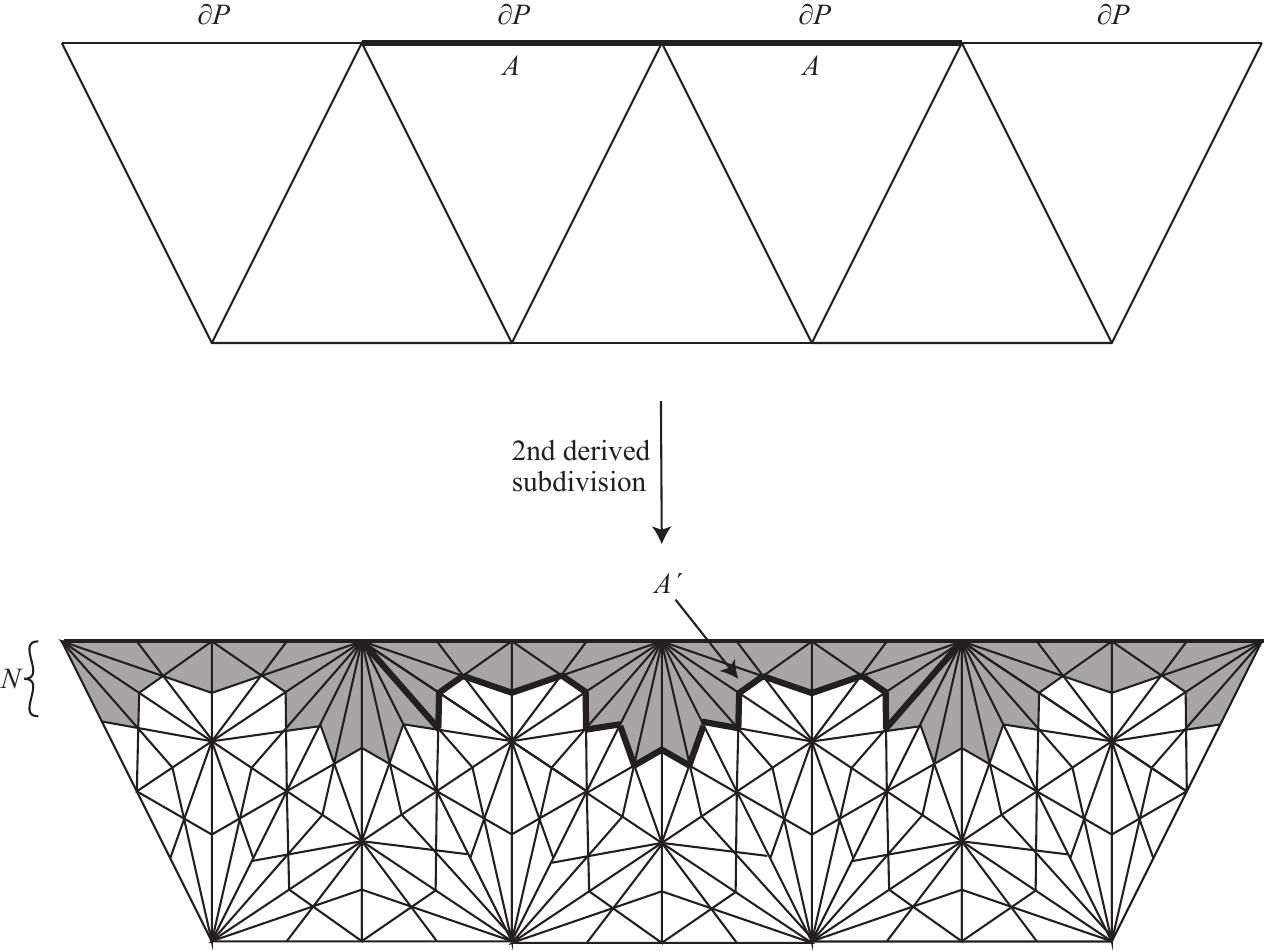}
  \caption{Pushing an arc into the interior}
  \label{Fig:PushArcsIntoInterior}
\end{figure}

We now turn to the proof.

\begin{restate}{Theorem}{Thm:DerivedSolidTorus}
Let $\calT$ be a triangulation of the solid torus $M$.
Then $M$ contains a core curve that is a subcomplex of $\calT^{(51)}$.
\end{restate}

\begin{proof}
We start with the triangulation $\calT$ of the solid torus $M$.
Let $\calH$ be its dual handle structure.
We give it its canonical affine structure, as in \refdef{CanonicalAffine}.
Each zero-handle is a truncated octahedron, which may be realised as a simplicial subset of the 2nd derived subdivision of the tetrahedron of $\calT$ that contains it.
Each one-handle of $\calH$ is realised as a product of a hexagon and an interval.
We collapse this vertically onto its co-core, which is a hexagonal face of the two incident truncated octahedra.

We now apply \refthm{CoreSolidTorusAffine}, which provides a core curve $C$. The intersection between $C$ and each hexagonal face is a collection of at most $24$ points. These may be made simplicial after $6$ barycentric subdivisions, by \reflem{PointsBecomeVertices}. Within each zero-handle, $C$ is a union of at most $48$ arcs, and these are simultaneously parallel to a collection of arcs $A$ in the boundary of the truncated octahedron.
We now make $A$ simplicial. The intersection between $A$ and each face of the truncated octahedron is at most $6 \times 48 = 288$ straight arcs. The ones that start and end on the boundary of the face come in at most $9$ arc types and these  can be made simplicial using at most $17$ barycentric subdivisions, by \reflem{SeveralArcTypes}. There are at most $24$ arcs that have at least one endpoint not on the boundary of the face. We make these simplicial using at most $24$ barycentric subdivisions using \reflem{ImproperArcsSimplicial}.  Finally, we can push these arcs in the boundary of the truncated octahedron into the interior and make them simplicial, using at most $2$ barycentric subdivisions, by \reflem{PushArcsIntoInterior}. In total, we have used at most $2 + 6 + 17 + 24 + 2 = 51$ subdivisions.
\end{proof}

\section{Nicely embedded handle structures}
\label{Sec:NicelyEmbedded}

\begin{definition}
Let $\calT$ be a triangulation of a compact three-manifold $M$, and let $\calH$ be the dual handle structure.
A subset of a zero-handle or one-handle $H$ of $\calH$ is \emph{subnormal} if it is obtained from $H$ by cutting along a collection of disjoint normal discs and then taking some of the resulting components.
\end{definition}

The proof of \refthm{LensSpaceCurve} (in the case where the lens space $M$ is not a prism manifold) proceeds by finding, within $M$, one of the solid tori $V$ in its Heegaard splitting embedded in a nice way.
More specifically, it has a handle structure where the union of the zero-handles and the one-handles is embedded in $M$ in the following way.

\begin{definition}
\label{Def:NicelyEmbeddedHandlebody}
Let $M'$ be a handlebody embedded in a three-manifold $M$.
Let $A$ be a union of disjoint annuli in $\bdy M'$.
Let $\calH'$ and $\calH$ be handle structures for $M'$ and $M$.
We say that $(\calH', A)$ is \emph{nicely embedded} in $\calH$ if the following hold:
\begin{itemize}
\item
$\calH'$ has only zero-handles and one-handles;
\item
each zero-handle of $\calH'$ is a subnormal subset of a zero-handle of $\calH$;
\item
each one-handle of $\calH'$ is a subnormal subset of a one-handle of $\calH$ and has the same product structure;
\item
the intersection between the annuli $A$ and any handle $H'$ of $\calH'$ is a union of components of intersection between $H'$ and handles of $\calH$. \qedhere
\end{itemize}
\end{definition}

\begin{definition}
\label{Def:KLNicelyEmbeddedHandlebody}
With notation as in the previous definition,
we say that $(\calH',A)$ is $(k, \ell)$--\emph{nicely embedded} in $\calH$, for non-negative integers $k$ and $\ell$,
if it is nicely embedded and, in addition, the following hold:
\begin{itemize}
\item
in any zero-handle of $\calH$, at most $k$ zero-handles of $\calH'$ lie between parallel normal discs in $\calH$;
\item
in any one-handle of $\calH$, at most $\ell$ one-handles of $\calH'$ lie between parallel normal discs in $\calH$. \qedhere
\end{itemize}
\end{definition}

\begin{definition}
\label{Def:NicelyEmbeddedManifold}
Let $M''$ be a three-manifold embedded in another three-manifold $M$.
Let $\calH''$ and $\calH$ be handle structures for $M''$ and $M$.
Let $\calH'$ be the handle structure just consisting of the zero-handles
and one-handles of $\calH''$, and let $A$ be the attaching annuli of the two-handles.
We say that $\calH''$ is $(k,\ell)$--\emph{nicely embedded} in $\calH$, for non-negative integers $k$ and $\ell$, if $(\calH',A)$ is $(k,\ell)$--nicely embedded in $\calH$.
\end{definition}

These definitions are designed to capture the essential properties of the handle structure that $M \cut S$ inherits when $S$ is a normal surface.
More specifically, we have the following.

\begin{lemma}
\label{Lem:NicelyEmbedded}
Suppose that $M$ is a compact three-manifold with triangulation $\calT$.
Suppose that $S$ is a normal surface, properly embedded in $M$.
Let $\calH$ be the handle structure dual to $\calT$ and $\calH'$ be the handle structure that $M \cut S$ inherits, but with the two-handles removed.
Let $A$ be their attaching annuli.
Then $(\calH', A)$ is nicely embedded in $\calH$. \qed
\end{lemma}

We typically collapse the one-handles of $\calH'$ vertically onto their co-cores.
Thus, the underlying manifold of $\calH'$ becomes a collection of balls, which are just its zero-handles, glued along discs in their boundary.

Once we have such a $(k, \ell)$--nice embedding, we get the following results.

\begin{theorem}
\label{Thm:NiceEmbeddedManifold}
Let $M$ be a compact three-manifold with a triangulation $\calT$.
Let $M'$ be a handlebody with a handle structure $\calH'$, and let $A$ be a union of disjoint annuli in $\bdy M'$.
Suppose that $M'$ is embedded in $M$ in such a way that $(\calH', A)$ is $(k, \ell)$--nicely embedded in the dual of $\calT$.
Then we can arrange that the following are all simplicial subsets of $\calT^{(m)}$:
\begin{itemize}
\item each zero-handle of $M'$;
\item each one-handle of $M'$, vertically collapsed onto its co-core;
\item the annuli $A$;
\end{itemize}
where $m = 17 + 2 \lceil \log_2 (2k+10) \rceil + \lceil \log_2(6+2\ell) \rceil + \lceil \log_2(4 + 2\ell) \rceil$.
\end{theorem}

\begin{theorem}
\label{Thm:NiceEmbeddedSolidTorus}
Let $M$ be a compact three-manifold with a triangulation $\calT$.
Let $V$ be a solid torus with a handle structure $\calH'$.
Suppose that $V$ is embedded in $M$ in such a way that $\calH'$ is $(k, \ell)$--nicely embedded in the dual of $\calT$.
Then there is a core curve of $V$ that is a subcomplex of $\calT^{(m+49)}$, where $m$ is as in \refthm{NiceEmbeddedManifold}.
\end{theorem}

\begin{proof}[Proof of Theorems~\ref{Thm:NiceEmbeddedManifold} and~\ref{Thm:NiceEmbeddedSolidTorus}]
We start with the triangulation $\calT$ of $M$.
Let $\calH$ be its dual handle structure.
We give it its canonical affine structure, as in the proof of \refthm{DerivedSolidTorus}.
Each zero-handle is a truncated octahedron, which may be realised as a simplicial subset of the second derived subdivision of the tetrahedron of $\calT$ that contains it.
Each one-handle is the product of a hexagon and an interval, but it is collapsed onto its hexagonal co-core.

Our goal is to construct an affine handle structure on $\calH'$.
Thus, each zero-handle of $\calH'$ is given the structure of a euclidean polyhedron.
We realise this as a polyhedron in the truncated octahedron of $\calH'$ that contains it and as a simplicial subset of a suitable iterated barycentric subdivision of $\calT$.

Now, within each zero-handle $H_0$ of $\calH$, every zero-handle of $\calH'$ is subnormal.
These subnormal zero-handles are obtained from the truncated octahedron $H_0$ by cutting along normal triangles and squares.
These are arranged into at most $4$ triangle types and at most one square type.
Our approach, in overview, is to arrange for the intersections between these normal discs and the one-handles of $\calH$ to be simplicial, then for the remainder of the boundary of these discs to be simplicial and then for the normal discs themselves to be simplicial.

Let us first focus on a one-handle of $\calH$, which is a product of a hexagon $X$ and an interval $[-1,1]$.
Within this one-handle, there are various one-handles of $\calH'$.
Only four possible one-handles of $\calH'$ in $X \times [-1,1]$ do not lie between parallel normal discs.
By assumption, at most $\ell$ one-handles do lie between parallel normal discs.
These one-handles are therefore obtained from $X \times [-1,1]$ by cutting along at most $6 + 2\ell$ normal discs and then possibly throwing away some components.
The union of these discs is of the form $\beta \times [-1,1]$ for normal arcs $\beta$ in $X$.
These arcs come in at most $3$ types.
Hence, by \reflem{SeveralArcTypes}, after at most $2 + \lceil \log_2(6+2\ell) \rceil$ barycentric subdivisions, we make $\beta$ simplicial in the triangulation of $X$.

Now consider each zero-handle $H_0$ of $\calH$.
We cut this handle along normal discs and then take some of the resulting components to get the subnormal zero-handles of $\calH'$ in $H_0$.
We are going to realise the boundary of these normal discs as simplicial in a suitable subdivision of the triangulation.
We have already arranged for their intersection with the one-handles to be simplicial.
Their intersection with each two-handle consists of at most $4 + 2\ell$ arcs, all of the same arc type, and so by using $1 + \lceil \log_2(4 + 2\ell) \rceil$ further barycentric subdivisions, we can also arrange for these arcs to be simplicial, by \reflem{TwoArcTypes}.

The number of normal discs that we need to consider within $H_0$ is at most $2k + 10$.
We now make these simplicial, using \reflem{SimplicialDiscsInPolyhedron}.
This requires at most $2 \lceil \log_2 (2k + 10)\rceil  + 8$ barycentric subdivisions.

Thus, each zero-handle of $\calH'$ is now a simplicial subset of $\calT^{(m)}$, where $m$ is as given in the statement of the theorem.
 Furthermore, when two zero-handles of $\calH'$ are joined by a one-handle, then we glue the simplicial subsets of $\calT^{(m)}$ corresponding to these zero-handles along some faces. These can be arranged to be flat euclidean convex polygons with at most $6$ sides. Thus, we have realised each one-handle, when vertically collapsed onto its co-core, as a simplicial subset of $\calT^{(m)}$. Also, the components of intersection between the two-handles and the zero-handles and between the two-handles and the collapsed one-handles are simplicial. Thus, we have proved \refthm{NiceEmbeddedManifold}.

Let us now prove \refthm{NiceEmbeddedSolidTorus}.
So $V$ is now a solid torus.
Each zero-handle and one-handle of $\calH'$ has the structure of a euclidean polyhedron, as described above.
This gives $\calH'$ an affine handle structure.
We can therefore apply \refthm{CoreSolidTorusAffine}, which gives a core curve $C$ of the solid torus $V$.

The intersection between $C$ and any one-handle of $\calH'$ is at most 24 arcs, which respect the product structure on the handle. When the one-handle is vertically collapsed onto its co-core, these arcs become points. Using \reflem{PointsBecomeVertices}, we can make these vertices in the triangulation of the co-core of the one-handle after $6$ barycentric subdivisions.

The intersection between $C$ and each zero-handle of $\calH'$ is a trivial tangle. Moreover, we have control over the arcs in the boundary of the handle to which it is parallel. Specifically, these arcs are a union of straight arcs in each face of the handle. These arcs come in two types: those that start and end on the boundary of the face, and those that have at least one endpoint in $C$. In each face (that arises as a component of intersection with the one-handles), there are at most $24$ of the latter type of arc. Hence, we need at most $24$ barycentric subdivisions to make these simplicial, by \reflem{ImproperArcsSimplicial}. There are at most $48 \times 6 = 288$ arcs in each face that start and end on the boundary of the face and these come in at most $9$ types. At most $17$ subdivisions are required to make these simplicial, by \reflem{SeveralArcTypes}.

Now in each zero-handle, $C$ runs parallel to these arcs, which have been made simplicial. Hence, using \reflem{PushArcsIntoInterior}, two further subdivisions are required to make $C$ simplicial.

The total number of barycentric subdivisions we have performed is at most $m + 49$.
\end{proof}

\section{Parallelity bundles}
\label{Sec:Parallelity}

In this section, we recall some material from~\cite[Section~5]{Lackenby:Composite} about parallelity bundles for handle structures.
Here we assume that $M$ is a compact orientable three-manifold.
We further assume that $\calH$ is a handle structure for $M$.

\begin{definition}
Suppose that $\gamma$ is a simple closed curve, properly embedded in $\bdy M$.
We say that $\gamma$ is \emph{standard} with respect to $\calH$ if it satisfies the following properties.
\begin{itemize}
\item
The curve $\gamma$ is disjoint from the two-handles of $\calH$,
\item
for each one-handle $H = D^1 \times D^2$ there is a finite set $P \subset \bdy D^2$ so that $\gamma \cap H = D^1 \times P$, and
\item
$\gamma$ meets at least one one-handle. \qedhere
\end{itemize}
\end{definition}

\begin{definition}
Suppose that $S \subset \bdy M$ is a subsurface.
We say that $\calH$ is a \emph{handle structure for the pair} $(M,S)$ if
\begin{itemize}
\item
$\calH$ is a handle structure for $M$ and
\item
the boundary of $S$ in $\bdy M$ is a union of standard curves for $\calH$.
\qedhere
\end{itemize}
\end{definition}

\begin{definition}
Suppose that $\calH$ is a handle structure for the pair $(M, S)$.
Suppose that $H$ is a zero-, one-, or two-handle of $\calH$.
We say that $H$ is a \emph{parallelity handle} if it admits a product structure $D^2 \times I$ so that
\begin{itemize}
\item
$H \cap S = D^2 \times \bdy I$ and
\item
for any other handle $H'$ of $\calH$ every component of $H \cap H'$ is \emph{vertical} in $H$: that is, of the form $\beta \times I$ where $\beta$ is an arc in $\bdy D^2$.
\qedhere
\end{itemize}
\end{definition}

Following the above definition we typically regard parallelity handles as $I$--bundles over $D^2$, their first coordinate.

\begin{definition}
The union of the parallelity handles in $\calH$ is the \emph{parallelity bundle} for $\calH$.
\end{definition}

By \cite[Lemma~5.3]{Lackenby:Composite} the $I$--bundle structures on the parallelity handles agree where they intersect and so give an $I$--bundle structure on the parallelity bundle.

\begin{definition}
Suppose that $\calB$ is an $I$--bundle over a surface $F$.
The resulting $(\bdy I)$--bundle over $F$ is $\bdy_h \calB$, the \emph{horizontal boundary} of $\calB$.
The resulting $I$--bundle over $\bdy F$ is $\bdy_v \calB$, the \emph{vertical boundary} of $\calB$.
The components of the boundary of $\bdy_h \calB$ (which equals the boundary of $\bdy_v \calB$) are called the \emph{corner curves} of $\calB$.
\end{definition}

\begin{definition}
\label{Def:GeneralisedParallelityBundle}
Suppose that $\calH$ is a handle structure for the pair $(M,S)$.
Suppose that $\calB^+$ is a three-dimensional submanifold of $M$.
We say that $\calB^+$ is a \emph{generalised parallelity bundle} if
\begin{itemize}
\item
$\calB^+$ is an $I$--bundle over a compact surface;
\item
the horizontal boundary of $\calB^+$ is $\calB^+ \cap S$;
\item
$\calB^+$ is a union of handles of $\calH$;
\item
any handle in $\calB^+$ that intersects the vertical boundary of $\calB^+$ is a parallelity handle, where the $I$--bundle structure on the parallelity handle agrees with the $I$--bundle structure of $\calB^+$;
\item
for any $i$--handle lying in $\calB^+$ that is incident to $j$--handle, where $j > i$, the $j$--handle must also lie in $\calB^+$.
\qedhere
\end{itemize}
\end{definition}

Note that the parallelity bundle $\calB$ is itself a generalised parallelity bundle.

\begin{definition}
We say that a generalised parallelity bundle $\calB^+$ is \emph{maximal} if $\calB^+$ is not properly contained in another generalised parallelity bundle.
\end{definition}

Note that there is always a maximal generalised parallelity bundle $\calB^+$ that contains all of $\calB$.
However the inclusion of $\calB$ into $\calB^+$ need not respect the $I$--bundle structure on all components of $\calB$.

\begin{lemma}
\label{Lem:BundleBoundaryCurves}
Let $\calB$ be the parallelity bundle and let $\calB^+$ be any maximal generalised parallelity bundle that contains $\calB$.
Then every corner curve of $\calB^+$ is a corner curve of $\calB$.
\end{lemma}

\begin{proof}
By hypothesis, $\calB^+$ contains $\calB$.
Hence, $\bdy_h \calB^+ = \calB^+ \cap S$ contains $\bdy_h \calB = \calB \cap S$.
Suppose that $\gamma$ is a corner curve of $\calB^+$.
So $\gamma$ is a component of $\bdy A$ for some component $A$ of $\bdy_v \calB^+$.
By definition, the handles of $\calB^+$ incident to $A$ are parallelity handles.
The $I$--bundle structures on $A$ and $\calB$ agree;
also any parallelity handle incident to $A$ lies in $\calB$.
Hence, $A$ is a component of $\bdy_v \calB$.
We deduce that $\gamma$ is a corner curve of $\calB$.
\end{proof}

\begin{definition}
Suppose that $G$ is an annulus, properly embedded in $M$, with boundary in $S$.
Suppose that $G'$ is an annulus in $\bdy M$ with $\bdy G = \bdy G'$.
Suppose also that $G \cup G'$ bounds a three-manifold $P$ such that
\begin{itemize}
\item
either $P$ is a parallelity region between $G$ and $G'$ or $P$ lies in a three-ball;
\item
$P$ is a non-empty union of handles;
\item
$\closure(M - P)$ inherits a handle structure from $\calH$;
\item
any parallelity handle of $\calH$ that intersects $P$ lies in $P$;
\item
$G$ is a vertical boundary component of a generalised parallelity bundle lying in $P$;
\item
$G' \cap (\bdy M - S)$ is either empty or a regular neighbourhood of a core curve of the annulus $G'$.
\end{itemize}
Removing the interiors of $P$ and $G'$ from $M$ is called an \emph{annular simplification}.
\end{definition}

The resulting three-manifold $M'$, obtained from an annular simplification, is homeomorphic the original manifold $M$.
This holds even in the case where $P$ is homeomorphic to the exterior of a non-trivial knot;
in this case $P$ lies in a three-ball in $M$.
We now restate~\cite[Proposition~5.6]{Lackenby:Composite}.

\begin{theorem}
\label{Thm:IncompressibleHorizontalBoundary}
Suppose that $M$ is a compact orientable irreducible three-manifold.
Suppose that $F$ is an incompressible subsurface of $\bdy M$.
Let $\calH$ be a handle structure for $(M, F)$.
Suppose that $\calH$ admits no annular simplification.
Let $\calB^+$ be any maximal generalised parallelity bundle in $\calH$.
Then the horizontal boundary of $\calB^+$ is incompressible.
\qed
\end{theorem}

\section{Finding a simplicial core curve of a lens space}
\label{Sec:ProofMain}

This section is devoted to the proof of \refthm{LensSpaceCurve}:
that is, for any triangulation $\calT$ of any lens space $M$ (except $\RP^3$) there is a relatively short curve with complement a solid torus or the orientation $I$--bundle over the Klein bottle.

Let $\calH$ be the handle structure of $M$ that is dual to $\calT$.
Let $S$ be an almost normal Heegaard torus in $\calT$, which exists by \refthm{AlmostNormalHeegaard}.
Cutting $M$ along $S$ gives two solid tori $X_1$ and $X_2$.
As explained in \refsec{HandleStructures}, these inherit handle structures $\calH_1$ and $\calH_2$.

We now consider some particular situations where the proof of \refthm{LensSpaceCurve} is fairly straightforward.
They highlight the approach that needs to be taken in the general case.

Suppose, as a special case, that one of $\calH_1$ or $\calH_2$ is $(0,0)$--nicely embedded within $\calH$.
We could then use \refthm{NiceEmbeddedSolidTorus} to find a core curve of one of the solid tori $X_1$ or $X_2$ that is simplicial in a suitable iterated barycentric subdivision of $\calT$.
But in general, the embeddings of $\calH_1$ and $\calH_2$ in $\calH$ are not $(0,0)$--nicely embedded,
because when $S$ contains two normal discs of the same type, the space between them becomes a zero-handle of $\calH_1$ or $\calH_2$ that violates the definition of a $(0,0)$--nice embedding.

For $i = 1$ and $2$, let $\calB_i$ be the parallelity bundle for $\calH_i$.
Suppose, as a different special case, that this is an $I$--bundle over a collection of discs, for $i=1$ or $2$. Then we could create a new handle structure from $\calH_i$ by removing $\calB_i$ and replacing it by two-handles.
This new handle structure is then $(0,0)$--nicely embedded in $\calH$,
and we could then apply \refthm{NiceEmbeddedSolidTorus}.

Of course, though, there is no particular reason for $\calB_i$ to be $I$--bundles over discs.
But according to \refthm{IncompressibleHorizontalBoundary}, we can apply annular simplifications and then extend the parallelity bundle to a generalised parallelity bundle $\calB_i^+$ with horizontal incompressible boundary.
In fact, \refthm{IncompressibleHorizontalBoundary} is not immediately applicable, because if one sets $F$ in that theorem to be all of the Heegaard surface $S$, then it is not incompressible in the two solid tori $X_1$ and $X_2$.
But setting this aside for the moment, suppose that we could ensure that the  generalised parallelity bundle $\calB_i^+$ has horizontal incompressible boundary.
It cannot be all of $S$, and so it is a union of disjoint discs and annuli.
Hence, $\calB_i^+$ consists of $I$--bundles over discs, annuli and M\"obius bands.
We can replace any $I$--bundles over discs by two-handles and remove any $I$--bundles over annuli using annular simplifications.
Thus, if $\calB_i^+$ contains no $I$--bundles over M\"obius bands, then we end with a handle structure for one of the solid tori that is $(0,0)$--nicely embedded in $\calH$.
We could then apply \refthm{NiceEmbeddedSolidTorus}.

To fix the problem that $S$ is not incompressible in $X_1$ and $X_2$, we work with the pair $(X_i, F_i)$, where $F_i$ is a suitable subsurface of $S$.
This is obtained from $S$ by cutting along a curve or curves. One way of producing the required curves is via the following lemma.

\begin{lemma}
\label{Lem:ShortEssentialCurveInTorus}
Let $V$ be a solid torus with a handle structure $\calH$. Then there is a simple closed curve $C$ in $\bdy V$ satisfying the following conditions:
\begin{itemize}
\item it is standard in $\calH$;
\item it runs over each component of intersection between $\bdy V$ and the one-handles at most once;
\item it is essential in $\bdy V$ and non-meridional.
\end{itemize}
Suppose also that $D$ is a union of disjoint discs in $\bdy V$ such that
\begin{itemize}
\item it is a union of components of intersection between $\bdy V$ and the handles;
\item if $H$ and $H'$ are handles of $\calH$ where $H'$ has higher index than that of $H$, then whenever a component of
$\bdy V \cap H$ lies in $D$, so do all incident components of $\bdy V \cap H'$.
\end{itemize}
Then we may also ensure that $C$ is disjoint from $D$.
\end{lemma}

\begin{proof}
Any essential curve in $\bdy V$ may be isotoped to a standard one, by first pushing it off the two-handles and then making it vertical in the one-handles. We consider all standard simple closed curves that are non-zero and non-meridional in $H_1(\bdy V; \ZZ/2\ZZ)$. We let $C$ be such a curve that runs over the one-handles the fewest number of times.
Then if it runs over a component of intersection between a one-handle and $\bdy V$ more than once, we can modify it to reduce this number by $2$. This might create a disconnected one-manifold. But if so, then we just focus on one component. We can choose this component to be non-zero and non-meridional in $H_1(\bdy V; \ZZ/2\ZZ)$. Thus, under the assumption that $C$ runs over the one-handles the fewest number of times, we deduce that $C$ in fact runs over each component of intersection between the one-handles and $\bdy V$ at most once.

Consider now the case where there are also the discs $D$. Then we isotope any essential curve off $D$ and make it standard. We consider such a curve that is non-zero and non-meridional in $H_1(\bdy V; \ZZ/2\ZZ)$. Among all such curves disjoint from $D$, we let $C$ be one that runs over the one-handles the fewest number of times. The same argument as above then applies.
\end{proof}

\begin{lemma}
\label{Lem:RemoveThreeCurves}
Let $M$ be a compact three-manifold with a handle structure $\calH$. Let $\calB$ be the parallelity bundle for $(M, \bdy M)$. Let $C$ be a standard curve in $\bdy M$ that is disjoint from $\bdy_h \calB$. Let $C'$ be three parallel copies of $C$, and let $F$ be $\bdy M \cut N(C')$. Then the parallelity bundle for $(M,F)$ is equal to $\calB$.
\end{lemma}


\begin{proof}
Every handle $H$ of $\calB$ is, by definition, a parallelity handle for $(M, \bdy M)$ and so satisfies $H \cap \bdy M = H \cap \bdy_h B$. Since $C$ is disjoint from $\bdy_h \calB$, it misses $H$, and therefore $H \cap F = H \cap \bdy M$. So, $H$ is a parallelity handle for $(M, F)$.

Now consider a parallelity handle $H$ for $(M, F)$. 
If $H$ is a two-handle, it is disjoint from the standard curves $C'$, and therefore $H \cap \bdy M = H \cap F$. 
Hence, in this case, $H$ is a parallelity handle for $(M, \bdy M)$. 
Suppose now that $H$ is a one-handle. 
It has the form $D^2 \times I$, where $H \cap F = D^2 \times \bdy I$. 
Each component of intersection between $H$ and any other handle has the form $\beta \times I$ for an arc $\beta$ in $\bdy D^2$. 
There are two such components arising from the intersection between $H$ and the incident zero-handles. 
There may be a further one or two components, arising from the intersection with the two-handles. 
If $H$ does have two components of intersection with the two-handles, then $H \cap \bdy M = H \cap F$, and therefore $H$ is a parallelity handle for $(M, \bdy M)$. 
On the other hand, if $H$ has fewer than two components of intersection with the two-handles, then it intersects $\bdy M \cut F$ once or twice. 
Therefore, $C'$ runs along the one-handle once or twice. 
But $C'$ consists of three parallel copies of $C$, and therefore the number of times that it runs along this one-handle is a multiple of three. 
This is a contradiction. 
This completes the proof when $H$ is a one-handle. 
Now suppose that $H$ is a zero-handle. 
If it is disjoint from the one-handles, then $C'$ misses it, since $C'$ is standard. 
Hence, in this case, $H$ is a parallelity handle for $(M, \bdy M)$. 
On the other hand, if $H$ intersects a one-handle, then this is also a parallelity handle for $(M,F)$ and hence, as argued above, this one-handle has two components of intersection with the two-handles. 
Thus, $\bdy D^2 \times I$ consists of intersections with one-handles and two-handles in an alternating fashion around the annulus $\bdy D^2 \times I$.
In particular, $H \cap \bdy M = D^2 \times \bdy I$, and therefore $H$ is again a parallelity handle for $(M, \bdy M)$.
\end{proof}

We are now equipped to prove the theorem.

\begin{restate}{Theorem}{Thm:LensSpaceCurve}
Let $M$ be a lens space other than $\RP^3$.
Let $\calT$ be any triangulation of $M$.
Then there is a simple closed curve $C$ that is a subcomplex of $\calT^{(86)}$,
such that the exterior of $C$ is either a solid torus or a twisted $I$--bundle over a Klein bottle.
\end{restate}

\begin{proof}

Let $S$ be an almost normal Heegaard torus in $\calT$, which exists by \refthm{AlmostNormalHeegaard}.
By \refprop{FromAlmostNormalToNormal}, $S$ can be arranged to be normal in the barycentric subdivision $\calT^{(1)}$.
Let $\calH$ be the handle structure of $M$ that is dual to $\calT^{(1)}$.
Cutting $M$ along $S$ gives two solid tori $X_1$ and $X_2$.
These then inherit handle structures $\calH_1$ and $\calH_2$.

Let $\calB_i$ be the parallelity bundle for $(X_i, \bdy X_i)$ with handle structure $\calH_i$.
Then $\bdy_v \calB_i$ is a collection of properly embedded annuli in $X_i$.
Their boundary curves are a (possibly empty) collection of essential curves on
$\bdy X_i$, each with
the same slope $\alpha_i$, together with some inessential curves on $\bdy X_i$.
We let $\alpha_i= \emptyset$ if there are no essential curves.

We consider three cases:
\begin{enumerate}
\item $\alpha_1$ or $\alpha_2$ is empty;
\item $\alpha_1$ and $\alpha_2$ are equal and non-empty;
\item $\alpha_1$ and $\alpha_2$ are distinct and non-empty.
\end{enumerate}

\begin{case}
\label{Case:Empty}
$\alpha_1$ or $\alpha_2$ is empty.
\end{case}

Say that $\alpha_i$ is empty. We claim that $\bdy_h \calB_i$ is a subsurface of $S$ that lies in a collection of discs. Otherwise, $\bdy_h \calB_i$ contains a component $F$ that is $S$ minus some open discs. There cannot be another copy of $F$ in $S$ that is disjoint from $F$, and so $F$ must be the horizontal boundary of a twisted $I$--bundle component of $\calB_i$. The zero section of this $I$--bundle is therefore a non-orientable surface. By attaching annuli to its boundary that lie in $\bdy_v \calB_i$ and then capping off with discs, we obtain a closed non-orientable surface in the solid torus. This does not exists, which proves the claim.

We may take the discs $D$ that contain $\bdy_h \calB_i$ to be as in \reflem{ShortEssentialCurveInTorus}. Specifically, each boundary component of $\bdy_h \calB_i$ bounds a disc in $S$, and we set $D$ to be the union of these discs (which may be nested).
Hence, there is a standard curve $C_i$ in $S$ that misses these discs, runs over each component of intersection between $\bdy X_i$ and the one-handles of $\calH_i$ at most once and is essential in $\bdy X_i$ and non-meridional. Let $C'_i$ be three parallel copies of $C_i$ and let $F_i = S \cut N(C_i')$. By \reflem{RemoveThreeCurves}, the parallelity bundle for $(X_i, F_i)$ is exactly $\calB_i$.

We now apply as many annular simplifications to $\calH_i$ as possible, forming a handle structure $\calH_i'$ for a solid torus isotopic to $X'_i$ with subsurface $F'_i$ of $\bdy X'_i$. This process does not introduce parallelity handles.
Thus, the horizontal boundary of the parallelity bundle $\calB'_i$ for $\calH'_i$ also lies in a union of disjoint discs in $\bdy X'_i$. We now extend this parallelity bundle to a maximal generalised parallelity bundle $\calB_i^+$. By \reflem{BundleBoundaryCurves}, the boundary curves of $\bdy_h \calB_i^+$ form a subset of the boundary curves of $\bdy_h \calB_i'$. They are therefore also inessential in $\bdy X'_i$. By \refthm{IncompressibleHorizontalBoundary}, $\bdy_h \calB_i^+$ is incompressible in $X'_i$. Hence it is a union of discs. So, $\calB_i^+$ consists of $I$--bundles over discs. Remove these discs and replace them by two-handles. The resulting handle structure is $(0,0)$--nicely embedded in $\calH$. So, by \refthm{NiceEmbeddedSolidTorus}, a core curve of $X'_i$ is simplicial in $\calT^{(79)}$. This proves the theorem in this case.

\begin{case}
\label{Case:EqualNonEmpty}
$\alpha_1$ and $\alpha_2$ are equal and non-empty.
\end{case}

For some $i \in \{ 1,2 \}$, $\alpha_i$ is non-meridional in $X_i$, since $M$ is not $S^2 \times S^1$. Fix some such $i$.
Let $C_i$ be a boundary component of $\bdy_v \calB_i$
with slope $\alpha_i$. Then isotope $C_i$ a little away from $\bdy_h \calB_i$ so that it becomes a standard curve in $\calH_i$.
Let $C'_i$ be three parallel copies of $C_i$ and let $F_i$ be $S \cut N(C_i')$.
We view $\calH_i$ as a handle structure for $(X_i, F_i)$. Again the parallelity bundle for $(X_i, F_i)$ is $\calB_i$, by \reflem{RemoveThreeCurves}.
Perform as many annular simplification to $\calH_i$ as possible, forming the three-manifold $X'_i$ with subsurface $F'_i$ of $\bdy X'_i$. Let $\calB_i'$ be its parallelity bundle. The boundary curves of its horizontal boundary therefore are inessential in $\bdy X'_i$ or have slope $\alpha_i$. Then, since $\alpha_i$ is non-meridional in $X_i$, \refthm{IncompressibleHorizontalBoundary} gives that we may extend $\calB'_i$ to a generalised parallelity bundle  $\calB_i^+$
with horizontal boundary that is incompressible in $X_i'$. The boundary curves of $\bdy_h \calB_i^+$ are inessential or have slope $\alpha_i$, by \reflem{BundleBoundaryCurves}.

Each component of $\calB_i^+$ is an $I$--bundle over a disc, annulus or M\"obius band. In fact, no component is an $I$--bundle over an annulus, for the following reason. A vertical boundary component of such an $I$--bundle is an incompressible annulus properly embedded in the solid torus $X'_i$. It is therefore boundary parallel in $X'_i$. By picking the vertical boundary component of the $I$--bundle appropriately, we may assume that the product region between it and an annulus in $\bdy X'_i$ contains the $I$--bundle. Hence, we could have performed an annular simplification, contradicting our assumption that as many of these as possible were performed.

Suppose that no component of $\calB_i^+$ is an $I$--bundle over a M\"obius band. Then we can replace $\calB_i^+$ by two-handles as in \refcase{Empty}.
The resulting handle structure is $(0,0)$--nicely embedded in $\calH$. So, by \refthm{NiceEmbeddedSolidTorus}, a core curve of $X'_i$ is simplicial in $\calT^{(79)}$.

So we are left with the case where some component of $\calB_i^+$ is an $I$--bundle over a M\"obius band. Then let $j = 3 - i$. Then $\alpha_j$ cannot be meridional in $X_j$, because we could patch a meridian disc for $X_j$ onto a M\"obius band in $X_i$ to get an embedded projective plane. This would imply that $M$ is $\RP^3$, contrary to assumption. Note that here we are using the fact that $\alpha_1$ and $\alpha_2$ are equal. Thus, the above argument gives a maximal generalised parallelity bundle $\calB_j^+$ for $(X'_j, F'_j)$ with $\bdy_h \calB_j^+$ incompressible in $X_j'$.

The only situation where the theorem is not proved in this case is when the generalised parallelity bundles $\calB_i^+$ and $\calB_j^+$ both contain components $B_i$ and $B_j$ that  are $I$--bundles over M\"obius bands. Then, $\alpha_i$ and
$\alpha_j$ are the boundaries of M\"obius bands in $X_i$ and
$X_j$ and, after an isotopy, these can be glued to form an embedded
Klein bottle in $M$. The exterior of this Klein bottle is a solid torus $V$. This is a general fact about Klein bottles in lens spaces (see the proof of \reflem{EmbeddedKleinBottle}).
But it can be seen directly as follows.
The exterior of $B_i$ in $X'_i$ is a solid torus which winds twice along $X'_i$.
Moreover, $\bdy_v B_i$ is longitudinal in this solid torus.
Similarly, the exterior of $B_j$ in $X'_j$ is a solid torus. If the two annuli $\bdy X'_i \cut \bdy_h B_i$ and $\bdy X'_j \cut \bdy_h B_j$ are isotoped
to be equal, then when $X'_i \cut B_i$ and $X'_j \cut B_j$ are glued
along this annulus, the result is a solid torus that is isotopic to $V$.
Each boundary component of $\bdy_v B_i$ is isotopic to a core curve of $V$.
In this case, we take the required curve $C$ to be one of these boundary components. Let $A$ be $\bdy_v \calB_i^+$ lying in the boundary of the manifold $M' = X'_i \cut \calB_i^+$. Then $(M', A)$, with its inherited handle structure, is $(0,0)$--nicely embedded in $\calH$.
By \refthm{NiceEmbeddedManifold}, $A$ is simplicial in $\calT^{(30)}$.
In particular, $C$, which is a boundary component of $A$ is simplicial in $\calT^{(30)}$, as required.

\begin{case}
\label{Case:UnequalNonEmpty}
$\alpha_1$ and $\alpha_2$ are distinct and non-empty.
\end{case}

For $i \in \{ 1, 2 \}$, let $C_i$ be some boundary curve of $\bdy_v \calB_i$ with slope $\alpha_i$, isotoped a little so that it becomes a standard curve. Note that this isotopy pushes $C_i$ into handles of $\calH_i$ that are not parallelity handles.

Now let $F$ be $S \cut N(C_1 \cup C_2)$. Since this lies in the complement of $C_1$ and $C_2$, which are essential curves with the distinct slopes,
$F$ is a union of discs. In particular, it is incompressible in $X_1$ and $X_2$.

Pick some $i$ and perform as many annular simplifications to $(X_i, F)$ as possible, giving a
handle structure $\calH'_i$ for a pair $(X'_i, F')$ that is isotopic to $(X_i,F)$.
Let $\calB_i^+$ be a maximal generalised parallelity bundle for $(X'_i, F')$ that contains its parallelity bundle.
By \refthm{IncompressibleHorizontalBoundary}, the horizontal boundary
of $\calB^+_i$ is incompressible. Since it is a subsurface of $F$, it is a union of discs.
Thus, $\calB^+_i$ consists of $I$--bundles over discs. Replace
each of these with a two-handle, giving a handle structure $\calH'$.

We claim that $\calH'$ is $(10,6)$--nicely embedded in $\calH$.
Because $\calH'$ is obtained from $\calH$ by cutting along the normal surface $S$,
it is $(k, \ell)$--nicely embedded in $\calH$ for some $k$ and $\ell$, but we need to show why we can take $k = 10$ and $\ell = 6$.
Consider a zero-handle $H'_0$ of $\calH'$ that lies between normally parallel discs. Let $H_0$ be the zero-handle of $\calH$ that contains it.
Since $H'_0$ does not lie in $\calB_i^+$, it is not a parallelity handle for $(X'_i, F')$. Hence, it intersects $C_1$ or $C_2$ in $S$.
Hence, in a component of $H_0 \cut S$ incident to $H'_0$, there must be a zero-handle of $\calH'$ that does not lie between parallel normal discs of $S$. Thus, at most $10$ zero-handles of $\calH'$ in $H_0$ do not lie between parallel normal discs. This number $10$ is twice the number of normal disc types of $S$ that can simultaneously exist within $H_0$.
Therefore, we may set $k = 10$. A similar argument gives that $\ell = 6$ also suffices.

Since $\calH'$ is $(10,6)$--nicely embedded in $\calH$,
\refthm{NiceEmbeddedSolidTorus} gives that there is a core curve of $X'_i$ that is simplicial in $\calT^{(86)}$. This proves the theorem in this case.
\end{proof}


We now use \refthm{DerivedSolidTorus} and \refthm{LensSpaceCurve} to prove our main technical result.

\begin{restate}{Theorem}{Thm:LensSpaceCore}
Let $M$ be a lens space, which is neither a prism manifold nor a copy of $\RP^3$.
Let $\calT$ be any triangulation of $M$.
Then the iterated barycentric subdivision $\calT^{(86)}$ contains a core curve of $M$ in its one-skeleton.
Furthermore, $\calT^{(139)}$ contains in its one-skeleton the union of the two core curves.
\end{restate}

\begin{proof}
Let $\calT$ be a triangulation of a lens space $M$, other than a prism manifold or $\RP^3$.
Note that the lens spaces that contain an embedded Klein bottle are exactly the prism manifolds by \reflem{EmbeddedKleinBottle}.
So, by \refthm{LensSpaceCurve}, there is a core curve $C$ of $M$ that is simplicial in $\calT^{(86)}$.
A regular neighbourhood of $C$ is simplicial in $\calT^{(88)}$.
Removing the interior of this regular neighbourhood
gives a triangulated solid torus. By \refthm{DerivedSolidTorus}, this contains a core curve that is
simplicial in its 51st barycentric subdivision. Hence, $M$ contains the union of its two core curves as a
simplicial subset of $\calT^{(139)}$.
\end{proof}

\section{Elliptic manifolds}
\label{Sec:Elliptic}

A three-manifold $M$ is \emph{elliptic} if there is a subgroup $\Gamma \subset \SO(4)$, acting freely on the three-sphere, with $M \homeo S^3/\Gamma$.
Since the universal cover of $M$ is the
three-sphere, the fundamental group $\pi_1(M)$ is necessarily finite.
Since all elements of $\SO(4)$ are orientation preserving, all elliptic
manifolds are orientable.  To give names to all of the elliptic
manifolds we invoke a beautiful result of Seifert and
Threlfall~\cite[page~568]{SeifertThrelfall33}.  See
also~\cite[Corollary 4.4.11]{Thurston97}.

\begin{theorem}
\label{Thm:Seifert}
Every elliptic manifold admits a Seifert fibred structure.
\end{theorem}



Recall that a Seifert fibred structure $\calF$ on a three-manifold $M$ is a foliation by circles where each circle $C$ has a neighbourhood $U$ that is a fibred solid torus.
The fibre $C$ has \emph{Seifert invariant} $q/p$ if every fibre in $\calF|\bdy U$ has slope $q/p$ with respect to a framing $\subgp{\lambda, \mu}$.
Here $\mu$ is the meridian of $U$ and $p$ is necessarily non-zero.
Note that replacing $\lambda$ by $\lambda + \mu$ changes the slope of the fibre to be $(q - p)/p$.

We simplify the notation $(M, \calF)$ to just $M$ when the foliation is understood.
We say that a fibre $C \subset M$ is \emph{generic} if it has integral Seifert invariant  (in other words, $|p| = 1$) and \emph{critical} otherwise.
If we quotient all fibres of $\calF$ to points we get the \emph{base orbifold} $B = M/S^1$, where critical fibres project to orbifold points.

The quotient induces a surjection $\pi_1(M) \to \pi_1^\orb(B)$ where the kernel is cyclic (possibly trivial) and generated by the generic fibre.
Since $\pi_1(M) \isom \Gamma$ is finite, we deduce that $\pi_1^\orb(B)$ is also finite.
Here, then, are the possibilities for $B$.

\begin{lemma}
\label{Lem:AllYourBaseAreBelongToUs}
Suppose $B$ is a closed two-dimensional orbifold (without mirrors).
Then $\pi_1^\orb(B)$ is finite if and only if $\chi^\orb(B) > 0$.
Thus $B$ is one of the following.
\begin{itemize}
\item
Cyclic -- $S^2$, $S^2(p)$, $P^2$, $S^2(p,q)$.
\item
Dihedral -- $P^2(r)$, $S^2(2,2,r)$, where $r \geq 2$.
\item
Tetrahedral -- $S^2(2,3,3)$.
\item
Octahedral -- $S^2(2,3,4)$.
\item
Icosahedral -- $S^2(2,3,5)$.
\end{itemize}
The names indicate the orbifold fundamental group. \qed
\end{lemma}

We call the last three the \emph{platonic} orbifolds.
In these cases we say the \emph{type} of $B$ is the orbifold fundamental group $\calP = \pi_1^\orb(B)$.
Here $\calP$ is one of the three platonic groups $A_4$, $S_4$ or $A_5$, respectively.

\begin{lemma}
\label{Lem:LensPrismBase}
Suppose $(M, \calF)$ is Seifert fibred and $B = M/S^1$ is the base orbifold.
If $B$ is cyclic or dihedral then $M$ is $S^2 \times S^1$, a lens space or prism manifold.
\end{lemma}

\begin{proof}
Suppose that $B = M/S^1$ is a two-sphere with at most two orbifold points.
Let $\alpha$ be a loop in $B$ cutting $B$ into a pair of discs each with at most one marked point.
Since $M$ is orientable and since $\alpha$ is two-sided, the preimage of $\alpha$ in $M$ is a torus $T$.
The preimage of either component of $B - \interior(N(\alpha))$ is a fibred solid torus and thus $M$ is a lens space or $S^2 \times S^1$.

Suppose that $B$ is a projective plane with at most one orbifold point.
Let $\alpha$ be a loop in $B$ reversing orientation.
So $B - \interior(N(\alpha))$ is a disc with at most one orbifold point.
Since $M$ is orientable, and since $\alpha$ is one-sided, the preimage of $\alpha$ in $M$ is a Klein bottle.
It follows that $M$ is a prism manifold.

Suppose that $B = S^2(2,2,r)$.
Let $\alpha$ be an arc connecting the orbifold points of order $2$.
Again, the preimage of $\alpha$ is a Klein bottle and $M$ is a prism manifold.
\end{proof}

We say that a Seifert fibred space $(M, \calF)$ is \emph{platonic} if $B = M/S^1$ is platonic.
In this case, the \emph{type} of $M$ is the same as the type of $B$.
Now Lemmas~\ref{Lem:AllYourBaseAreBelongToUs} and~\ref{Lem:LensPrismBase} allow us to coarsely name the elliptic manifolds.

\begin{proposition}
\label{Prop:Names}
Every elliptic manifold is either a lens space, a prism manifold, or a platonic manifold. \qed
\end{proposition}

To pin down an elliptic manifold precisely, we must discuss how the Seifert invariants of the critical fibres interact with the \emph{Euler number}.
Suppose, for this paragraph, that $M$ is a platonic manifold.
Suppose the Seifert invariants of the critical fibres $C_i$ are $q_i/p_i$, for $i = 1, 2, 3$.
Following Orlik~\cite[page~91]{Orlik72}, there is an integer $q$ so that $\pi_1(M)$ has the following presentation.
\[
\pi_1(M) = \group{a_1, a_2, a_3, b}{[a_i,b] = 1, a_i^{p_i} = b^{-q_i},
a_1 a_2 a_3 = b^q}
\]
Here $\subgp{b}$ is the central subgroup and is generated by a generic fibre in $M$.
Killing $b$ gives $\pi_1^\orb(B)$.
The integer $q$ is determined by the \emph{Euler number} of $(M, \calF)$:
\[
e(M, \calF) = - q - \sum \frac{q_i}{p_i}.
\]
In a small abuse of notation we call the data $(q, q_1/p_1, q_2/p_2, q_3/p_3)$ the \emph{Seifert invariants} of $M$.
In the terminology of Orlik~\cite[page~88]{Orlik72}, this manifold would be denoted by
\[
\{ q, (o_1, 0); (q_1,p_1), (q_2,p_2), (q_3,p_3) \}.
\]

Suppose $U_i$ is a fibred neighbourhood of the critical fibre $C_i \subset M$.
If we change the framing of $\bdy U_i$, say by replacing $q_i/p_i$ with $(q_i + p_i)/p_i$, then in order to keep the Euler number the same we must also change the framing about a generic fibre, by replacing $q$ with $q - 1$.
We call this process \emph{reframing}.

Note also that there is an orientation-reversing homeomorphism between the manifolds having Seifert invariants
\[
(q,q_1/p_1, \dots, q_n/p_n)
\quad \mbox{and} \quad
(-q,-q_1/p_1, \dots, -q_n/p_n).
\]
Since we are not considering our manifolds as oriented in this paper, we view these two Seifert fibre spaces as equivalent.

We now record a very useful observation, essentially due to Seifert and Threlfall~\cite[page~573]{SeifertThrelfall33}.

\begin{proposition}
\label{Prop:PlatonicHomeo}
Suppose $M$ is a platonic manifold with $B = M/S^1 = S^2(p_1, p_2, p_3)$.
Then the Seifert invariants of $M$ can be recovered, up to reframing and reversing orientation, from $B$ and the order of $H_1(M)$.
\end{proposition}

\begin{proof}
Abelianizing the fundamental group shows $H_1(M)$ has order
\[
\big| q p_1 p_2 p_3 + q_1 p_2 p_3 + p_1 q_2 p_3 + p_1 p_2 q_3 \big|.
\]
A bit of modular arithmetic
shows the data $q, q_1, q_2, q_3$ can be recovered, up to reframing and changing all their signs, from $B = M/S^1$ and the order of $H_1(M)$.

For example, suppose that $B$ is $S^2(2,3,4)$.
By applying an orientation-reversing homeomorphism if necessary, we may assume that $e(M, \calF) \leq 0$.
Its Seifert invariants are $(q,1/2,q_2/3,q_3/4)$ where $q_2 \in \{ 1,2 \}$ and $q_3 \in \{1,3 \}$.
Note that
\[
|H_1(M)|= |24 q + 12+ 8q_2 + 6q_3| = 24 q + 12+ 8q_2 + 6q_3
\]
since $e(M, \calF) \leq 0$.
We may compute $q_2$ using the fact that $|H_1(M)| \equiv 8 q_2$ mod $3$, and we may compute $q_3$ using the fact that $|H_1(M)| \equiv 6 q_3 + 12$ mod $8$.
Then finally, $q$ can be determined using the above equality for $|H_1(M)|$.

The cases of the other platonic manifolds are similar.
\end{proof}

\begin{remark}
It is evident that the procedure given in the above proof runs in polynomial time as a function of the number of digits of $|H_1(M)|$.
\end{remark}

\begin{lemma}
\label{Lem:Platonic}
Suppose $M$ is a platonic manifold with base orbifold $B = M/S^1$.
The map $\pi_1(M) \to \pi_1^\orb(B)$ is the only surjection of the fundamental group of $M$ onto a platonic group, up to post-composing by an automorphism of $\pi_1^\orb(B)$.
\end{lemma}

\begin{proof}
Suppose $G = \pi_1(M)$ and $\calP = \pi_1^\orb(B)$.
Let $\rho \from G \to \calP$ be the associated surjection.
Now suppose $\rho' \from G \to \calP'$ is any surjection to a platonic group.
Since $\calP'$ has trivial centre, the map $\rho'$ kills the fibre of $M$.
Thus $\rho$ factors $\rho'$.
However, there is no nontrivial surjective map between distinct platonic groups.
Thus $\rho = \rho'$, perhaps after applying an automorphism of $\calP$.
\end{proof}

\begin{proposition}
\label{Prop:PrismCriterion}
Suppose $q/p$ and $s/r$ are Farey neighbours, with $q \geq 2$.
An orientable three-manifold $M$ is homeomorphic to the prism manifold $P(p,q)$ if and only if $M$ is double covered by the lens space $L = L(2pq, ps+qr)$ and the subgroup of $\pi_1(L)$ fixed by the deck group action has order $2p$.
\end{proposition}


\begin{proof}
The forward direction is the statement of \reflem{PrismCover}.

For the backwards direction let $G = \pi_1(L) \isom H_1(L)$.
Since $M$ is double covered by a lens space, the resolution of the spherical space form conjecture~\cite{Perelman1, Perelman2, Perelman3} implies $M$ is elliptic.
Since $\pi_1(M)$ has an index two cyclic subgroup, namely $G$, we deduce that $\pi_1(M)$ cannot surject a platonic group.

\refprop{Names} implies $M$ is either a lens space or a prism manifold.
However, in any double cover of a lens space, the deck group acts trivially on homology.
We deduce that $M$ is a prism manifold.
According to \reflem{PrismCover}, the coefficients $p$ and $q$ can be recovered from the order and index, respectively, of the subgroup of $G$ that is fixed by the action of the deck group.
\end{proof}

\begin{proposition}
\label{Prop:PlatonicCriterion}
Suppose $\calP$ is one of the three platonic groups $A_4$, $S_4$ or $A_5$.
Let $d$ be the order of $\calP$.
An orientable three-manifold $M$ is homeomorphic to a platonic manifold of type $\calP$ if and only if it is $d$--fold covered by a lens space $L$ with deck group isomorphic to $\calP$.
Moreover, the Seifert invariants of $M$ can be recovered from the order of $H_1(M)$.
\end{proposition}

\begin{proof}
For the forward direction, recall the fundamental group of $M$ has the form
\[
G = \pi_1(M) = \group{a_1, a_2, a_3, b}{[a_i,b] = 1, a_i^{p_i} =
  b^{-q_i}, a_1 a_2 a_3 = b^q}.
\]
The subgroup $\subgp{b}$ is central.
The corresponding cover is an elliptic manifold with cyclic fundamental group, and so is a lens space.
The deck group is isomorphic to $G/\subgp{b} \isom \calP$.

For the backwards direction, since $M$ is covered by a lens space, the spherical space form conjecture (which is a consequence of the geometrisation conjecture, proved by Perelman \cite{Perelman1, Perelman2, Perelman3}) implies that $M$ is elliptic.
Since the degree of the cover equals the order of the deck group the covering is normal.
Thus $\pi_1(M)$ is a cyclic extension of $\calP$.
We deduce $M$ is not a lens space or prism manifold.

\refprop{Names} implies $M$ is a platonic manifold.
\reflem{Platonic} implies $M$ has the type of $\calP$.
Finally, \refprop{PlatonicHomeo} states that the Seifert invariants can be recovered from the order of $H_1(M)$.
\end{proof}

\section{Certifying \texorpdfstring{$T^2 \times I$}{T2 x I} and \texorpdfstring{$K^2 \twist I$}{K2 x I}}
\label{Sec:IBundles}

As usual, we assume that any three-manifold $M$ is given via a finite triangulation $\calT$.
The decision problem \textsc{Recognising $T^2 \times I$} takes $\calT$ as its input and it asks whether $M$ is homeomorphic to $T^2 \times I$.
The decision problem \textsc{Recognising $K^2 \twist I$} is defined similarly.
Both are dealt with by Haraway and Hoffman~\cite[Theorem~3.6]{HarawayHoffman19}.

\begin{theorem}
\label{Thm:RecognisingBundles}
The problems \textsc{Recognising $T^2 \times I$} and \textsc{Recognising $K^2 \twist I$} are in \NP.  \qed
\end{theorem}

There is a subtle point to note here.
Haraway and Hoffman, for their certificate for \textsc{Recognising $T^2 \times I$}, rely on~\cite[Theorem~12.1]{Lackenby:Knottedness}.
There, given a handle structure $\calH$ of a sutured manifold $(M, \gamma)$, the theorem provides a certificate that $(M, \gamma)$ is a product sutured manifold.
In our setting, we would simply check that $M$ had two toral boundary components $T_0$ and $T_1$, say, and would assign the sutured manifold structure $R_-(M) = T_0$ and $R_+(M) = T_1$.
Then~\cite[Theorem 12.1]{Lackenby:Knottedness} provides a certificate that $(M, \emptyset)$ is a product sutured manifold, which is equivalent to the statement that $M$ is homeomorphic to $T^2 \times I$.

However, there is a gap in the published proof of~\cite[Theorem~12.1]{Lackenby:Knottedness}.
There one tacitly assumes that $\gamma$ is non-empty.
Nevertheless, there is a straightforward fix.
We take as the certificate a non-separating annulus $A$ in normal form with respect to the triangulation $\calT$ and with weight at most an exponential function of the number of tetrahedra of $\calT$.
Such an annulus exists by~\cite[Lemma~8.5, Theorem~8.3]{Lackenby:Knottedness}. Then we form a handle structure on $M \cut A$ where the number of handles is bounded above by a linear function of the number of tetrahedra of $\calT$, using~\cite[Theorem~9.3]{Lackenby:Knottedness}.
This is a product sutured manifold if and only if $M$ was a product; however  now $M \cut A$ has a non-empty collection of sutures and so the proof of~\cite[Theorem~12.1]{Lackenby:Knottedness} applies.

We also note that there are alternative methods of certifying $T^2 \times I$.
One is to use the following result.

\begin{theorem}
Let $M$ be a compact orientable three-manifold with two boundary components, both of which are tori.
Let $s_1, s_2, s_3$ be slopes on one of the boundary components $T$ of $M$ that represent the three non-trivial elements of $H_1(T; \ZZ/2\ZZ)$.
Then $M$ is homeomorphic to $T^2 \times I$ if and only if the three manifolds obtained by Dehn filling along $s_1$, $s_2$, and $s_3$ are all homeomorphic to a solid torus. \qed
\end{theorem}

The proof is given inside the proof of~\cite[Theorem~11]{Haraway20}.
Thus we can certify that $M$ is homeomorphic to $T^2 \times I$ as follows.
We check that $M$ has two boundary components, both of which are tori.
Let $T$ be one of these.
We pick embedded normal one-manifolds $C_1$, $C_2$ and $C_3$ that represent the three non-trivial elements of $H_1(T; \ZZ/2\ZZ)$,
with the property that each $C_i$ intersects each triangle of $T$ in at most one normal arc.
Each $C_i$ consists of some parallel essential simple closed curves in $T$, plus possibly some curves that bound discs in $T$.
Let $C_i'$ be one of these essential curves.
The curves $C_1', C_2', C_3'$ again represent the three non-trivial elements of $H_1(T; \ZZ/2\ZZ)$.
We form triangulations for the three-manifolds obtained by Dehn filling along the slopes of $C_1', C_2', C_3'$ as follows.
We barycentrically subdivide the triangulation of $M$ once, so that each $C'_i$ is simplicial. Then
we attach (simplicially) a triangulated disc to $M$ along $C'_i$.
Finally we attach a triangulated three-cell to complete the Dehn filling.
Each of the resulting triangulations has at most $56$ times as many tetrahedra as the original, given, triangulation of $M$.
The final piece of our certificate is that these three triangulated three-manifolds are solid tori; for this we use~\cite[Corollary~2]{Ivanov08}.

There is even a third method of providing a certificate for \textsc{Recognising $T^2 \times I$}, going back to Schleimer's thesis~\cite[Chapter~6]{Schleimer01}.
Let $M$ be the given three-manifold equipped with the triangulation $\calT$.
We first check that $M$ has the homology of $T^2$.
As in~\cite[Theorem~15.1]{Schleimer11}, the first third of the certificate is a sequence $(\calT_i, v(S_i))_{i = 0}^n$ where:
\begin{itemize}
\item
$\calT_0 = \calT$,
\item
$v(S_i)$ is the normal vector of $S_i$, a fundamental non-vertex linking normal two-sphere in $\calT_i$ (for $i < n$),
\item
$v(S_n)$ is the zero-vector,
\item
$\calT_{i+1}$ is the result of \emph{crushing} $\calT_i$ along $S_i$~\cite[Section~13]{Schleimer11}, and
\item
$\calT_n$ is \emph{zero-efficient}~\cite[Definition~4.10]{Schleimer11}.
\end{itemize}
We next certify three-sphere components of $\calT_n$ and discard them to obtain $\calT'$.
The final third of the certificate follows the plan of~\cite[Theorem~6.2.1]{Schleimer01}.
We find a list $(T_i)_{i = 0}^{2n}$ of disjoint tori in $\calT'$.
Each even torus $T_{2k}$ (except for the first and last) is almost normal and \emph{normalises via isotopy} to the normal tori $T_{2k \pm 1}$.
The union $T_0 \cup T_{2n}$ is the frontier of the boundary of a small regular neighbourhood of $\bdy M$, taken in $\calT$.
We require that $T_0$ and $T_{2n}$ normalise via isotopy to $T_1$ and $T_{2n - 1}$, respectively.
That the $T_i$ exist and have controlled weight is proved in~\cite[Chapter~6]{Schleimer01}.
The normalisations are produced in polynomial time using the algorithm of~\cite[Theorem~12.1]{Schleimer11}.

There are several possible certificates for \textsc{Recognising $K^2 \twist I$}.
One is to exhibit a double cover $\cover{M}$ of the manifold $M$, to certify that $\cover{M}$ is $T^2 \times I$ and to check that $M$ has a single torus boundary component.
Specifically, the certificate is as follows:
\begin{enumerate}
\item a triangulation $\cover{\calT}$ of a three-manifold $\cover{M}$;
\item a simplicial involution $\phi$ of $\cover{\calT}$ with no fixed points;
\item a combinatorial isomorphism between $\cover{\calT} / \phi$ and $\calT$;
\item a certificate that $\cover{M}$ is homeomorphic to $T^2 \times I$.
\end{enumerate}
We now check that $\bdy M$ is a single torus and verify the given certificate.
This suffices because $K^2 \twist I$ is the unique orientable three-manifold with a single toral boundary component and that is double covered by $T^2 \times I$~\cite[Theorem~10.5]{Hempel}.

\section{Certificates for elliptic manifolds}
\label{Sec:CertificateElliptic}

In this section, we give our method for certifying whether a closed three-manifold is an elliptic manifold and, if it is, then which elliptic manifold it is.
That is, we prove the following.

\begin{restate}{Theorem}{Thm:Elliptic}
The problem \textsc{Elliptic manifold} lies in \NP.
\end{restate}

\begin{restate}{Theorem}{Thm:NamingElliptic}
The problem \textsc{Naming elliptic} lies in \FNP.
\end{restate}

In what follows we assume that the three-manifold $M$ is given via a finite triangulation $\calT$.

\subsection{Some problems in \NP}

In our proofs of Theorems~\ref{Thm:Elliptic} and~\ref{Thm:NamingElliptic}, we rely on the fact that the following decision problems lie in \NP.
\begin{enumerate}
\item \textsc{Recognising $S^3$} -- \cite[Theorem~15.1]{Schleimer11} and~\cite[Theorem~2]{Ivanov08}.
\item \textsc{Recognising $D^2 \cross S^1$} -- \cite[Corollary~2]{Ivanov08}.
\item \textsc{Recognising $T^2 \times I$} -- \cite[Theorem~3.6]{HarawayHoffman19}.
\item \textsc{Recognising $K^2 \twist I$} -- \cite[Theorem~3.6]{HarawayHoffman19}.
\end{enumerate}

\subsection{Some problems in {\textsc{P}}}

In what follows we rely on the algorithm of Kannan and Bachem~\cite{KannanBachem79} which places a given integer matrix into Smith normal form. 
The running time and also the bit-size of the output are bounded above by polynomial functions of the bit-size of the input.

We also use the fact, due to Burton~\cite[Corollary~8]{Burton11}, that it is decidable in polynomial time whether two triangulations of compact three-manifolds are combinatorially isomorphic. Burton proved this result by creating from a triangulation of a compact three-manifold an \emph{isomorphism signature}, which gives a labelling of the simplices of the triangulation, and it has the property that two triangulations are combinatorially isomorphic if and only if their isomorphism signatures are equal \cite[Theorem 7]{Burton11}. Furthermore, his algorithm also computes the automorphism group of a single triangulation in polynomial time.

We now provide various types of certificates for various types of elliptic manifolds.
In each case, we give the certificate and explain how it is verified.
The time to complete the verification is bounded above by a polynomial function of $|\calT|$, the number of tetrahedra in $\calT$.
Finally, we prove that such a certificate exists if and only if $M$ is the corresponding type of elliptic manifold.

In our certificates, we give more information than is strictly necessary. We do this in order to make the certificates easier to interpret by the reader of this paper. Some parts of the certificates could be removed, at the cost of requiring the checker of the certificate to do more work. We place an asterisk against those parts of the certificate that seem to crucial given the current state of knowledge.

\subsection{Three-sphere}

Recognising the three-sphere is discussed immediately above.

\subsection{Real projective space}
\label{Sec:RP3}

The certificate here is:
\begin{enumerate}
\item[(1)]
a triangulation $\cover{\calT}$ of a three-manifold $\cover{M}$;
\item[$\ast$(2)]
a certificate that $\cover{M}$ is the three-sphere;
\item [(3)]
a simplicial involution $\phi$ of $\cover{\calT}$ that has no fixed points; and
\item [(4)]
a simplicial isomorphism between $\cover{\calT} / \phi$ and $\calT$.
\end{enumerate}
The first, third, and fourth parts may be verified in polynomial time;
we omit the details.


By the spherical space form conjecture~\cite{Perelman1, Perelman2, Perelman3} or \cite{Livesay60}, the manifold $M$ has such a certificate if and only if it is $\RP^3$.
The output of the algorithm records whether the certificate has been verified, and in the function case, it also gives the Seifert data $(0, 1/2)$.

Only (2) appears to be necessary. Instead of (1), (3) and (4) being part of the certificate, one can proceed as follows. The verifier is given a triangulation $\calT$ of a three-manifold $M$. Using the algorithm of Kannan and Bachem, the verifier can compute $H_1(M)$ in polynomial time and check that it is $\mathbb{Z}/2\mathbb{Z}$. The verifier can also construct the unique non-trivial homomorphism $\pi_1(M) \rightarrow \mathbb{Z}/2\mathbb{Z}$, and hence can build the unique double cover $\cover{M}$ and its lifted triangulation $\cover{\calT}$. Using Burton's isomorphism signature, the verifier can ensure that the simplices of this triangulation are labelled in a canonical fashion. The certificate in (2) applies to this labelled triangulation. Once the certificate in (2) is verified, the verifier can deduce that $M$ is $\RP^3$.

\begin{remark}
There is an interesting generalisation of the above.
Suppose that $M = L(p, q)$ is given via a triangulation $\calT$;
suppose that $p$ is bounded by a uniform polynomial in $|\calT|$.
In this case, the verifier can construct the $p$--fold homology cover in polynomial time and also check that the deck group is cyclic.
They then check a (provided) certificate that the cover is the three-sphere.
This proves that $M$ is a lens space. 
Finally, and more delicately, they verify the coefficient $q$ in polynomial time using~\cite[Theorem~1.1]{Kuperberg18}.
\end{remark}

\subsection{Non-prism non-\texorpdfstring{$\RP^3$}{RP3} lens spaces}
\label{Sec:NonPrismNonRP3Lens}

The certificate is as follows:
\begin{enumerate}
\item [$\ast$(1)]
a simplicial subset $C$ of the one-skeleton of $\calT^{(86)}$;
\item [(2)]
a triangulation $\calT'$ of a three-manifold $X$;
this will in fact be the exterior of $C$;
\item [(3)]
a simplicial isomorphism between $\calT'$ and the triangulation that results from $\calT^{(88)}$ by keeping only those simplices that are disjoint from $C$;
\item [$\ast$(4)]
a certificate establishing that $X$ is a solid torus;
\item [(5)]
simplicial one-chains $\lambda$ and $\mu$ in $\calT^{(88)}$;
these will in fact be meridional and longitudinal curves on $\bdy N(C)$;
\item [(6)]
positive integers $p$ and $q$ with $1 \leq q < p$;
$M$ will in fact be homeomorphic to $L(p,q)$.
\end{enumerate}

To verify this certificate we must check the following:
\begin{enumerate}
\item that $C$ is a circle;
\item that the mapping given in part (3) is a simplicial isomorphism;
\item that the certificate in part (4) shows that $X$ is a solid torus.
\item that $\mu$ and $\lambda$ are one-cycles that generate $H_1(\bdy N(C))$;
\item that $\mu$ is trivial in $H_1(N(C))$ and that $\lambda$ is a generator;
\item that the kernel of the homomorphism $H_1(\bdy N(C)) \to H_1(X)$ (induced by inclusion) is generated by $p \lambda + q \mu$; and
\item that $(p,q) \neq (4p', 2p' \pm 1)$ for some integer $p'$.
\end{enumerate}
Once verified this also, for the function case, records $p$ and $q$. 
More specifically, since we are requiring that the output is the Seifert data for $M$, then this is $(0,q/p)$.

If we can verify the certificate, then $M$ is the manifold $L(p,q)$.
Conversely, suppose $M$ is the lens space $L(p, q)$ and is neither a prism manifold nor $\RP^3$.
Then $(p,q) \neq (4p', 2p' \pm 1)$ by \reflem{PrismIsLens}.
By \refthm{LensSpaceCore}, the subdivision $\calT^{(86)}$ contains a simplicial core curve $C$.
Its regular neighbourhood $N(C)$ is simplicial in $\calT^{(88)}$.
An essential simple closed curve in $\bdy N(C)$ that lies in the link of some vertex of $C$ gives a simplicial meridian $\mu$.
An embedded simplicial arc in $\bdy N(C)$ joining opposite sides of $\mu$ can then be closed up to form a simplicial longitude $\lambda$.
Thus, the required certificate exists.

It can be verified in polynomial time, as a function of the number of tetrahedra in $\calT$, for the following reasons.
The subdivision $\calT^{(88)}$ contains $24^{88}$ times as many tetrahedra as $\calT$.
Thus, steps (1)--(3) in the verification can be achieved in polynomial time.
For (4)--(6), we note that the boundary maps in the chain groups for $N(C)$ and $\bdy N(C)$ can be expressed as integer matrices, where the number of rows and columns is bounded by a constant multiple of the number of tetrahedra in $\calT^{(88)}$ and where each column has $L^1$ norm at most $3$.
Thus, (4)--(6) can be verified in polynomial time using linear algebra, as in the Kannan-Bachem algorithm.

One can alternatively just use (1) and (4) in the certificate. Instead of (2) and (3), the verifier can compute the triangulation for the exterior $X$ of $C$ obtained from $\calT^{(88)}$ by keeping only those simplices that are disjoint from $C$. Furthermore, by using Burton's isomorphism signature, the verifier can give the simplices of this triangulation a canonical labelling. 
Furthermore, the verifier can build the simplicial one-chains $\lambda$ and $\mu$ and compute the integers $p$ and $q$, using linear algebra.
Note that, since they are computed in polynomial time, $p$ and $q$ necessarily have polynomial bit-size.

\subsection{Prism lens spaces}
\label{Sec:PrismLens}

Here, the certificate is either as in the case of non-prism non-$\RP^3$ lens spaces (but where the integers $p$ and $q$ do satisfy $(p, q) = (4p', 2p' \pm 1)$ for some positive integer $p'$) or the following:
\begin{enumerate}
\item [$\ast$(1)] a simplicial subset $C$ of the one-skeleton of $\calT^{(86)}$;
\item [(2)] a triangulation $\calT'$ of a three-manifold $X$;
this will in fact be the exterior of $C$;
\item [$\ast$(3)] a certificate that $X$ is homeomorphic to $K^2 \twist I$;
\item [(4)] a triangulation $\cover{\calT}$ of a three-manifold $\cover{M}$;
this will in fact be the double cover of $M$ for which the inverse image of $K^2\twist I$ is a copy of $T^2 \times I$;
\item [(5)] a simplicial involution $\phi$ of $\cover{\calT}$ that has no fixed points;
\item [(6)] a simplicial isomorphism between $\calT$ and $\cover{\calT} / \langle \phi \rangle$;
\item [(7)] a simplicial subset $\tilde C$ of $\cover{\calT}^{(86)}$, partitioned into two subsets $\tilde C_1$ and $\tilde C_2$;
this will in fact be the inverse image of $C$ partitioned into its two components;
\item [(8)] a triangulation $\cover{\calT}'$ of a three-manifold $\tilde X'$;
this will be the exterior of $\tilde C_1$;
\item [$\ast$(9)] a certificate that $\tilde X'$ is a solid torus;
\item [(10)] simplicial one-chains $\lambda$ and $\mu$ in $(\cover{\calT}')^{(88)}$;
these will in fact be meridional and longitudinal curves on $\bdy N(\tilde C_1)$; and
\item [(11)] a positive integer $p$;
$\cover{M}$ will in fact be the lens space $L(2p, 1)$.
\end{enumerate}
We check that $C$ is a simple closed curve.
We check that $\calT'$ is the triangulation obtained from $\calT^{(88)}$ by removing those simplices incident to $C$.
We verify the certificate that this is a copy of $K^2 \twist I$, and so on.
These imply that $M$ is a prism manifold.
Once we have checked that $\cover{M}$ is the lens space $L(2p, 1)$, we then check that $\phi_\ast$ acts trivially on $H_1(\cover{M})$. By \reflem{PrismCover}, this implies that $M$ was the manifold $P(p,1)$, which is indeed both a prism manifold and a lens space.

The existence of this certificate is a consequence of \refthm{LensSpaceCurve}.

As in previous cases, some parts of this certificate can be dispensed with. One can alternatively compute $H_1(X)$ using the triangulation $\calT'$ provided in (2) and check that it is isomorphic to $\mathbb{Z} \times (\mathbb{Z}/2\mathbb{Z})$. One can then construct all double covers of $M$.
One of these will be the cover $\cover{M}$ with triangulation $\cover{\calT}$. The inverse image of $C$ in  $\cover{\calT}^{(86)}$ is readily computed. There are two possible choices for the component $C_1$. A triangulation of its exterior is readily computed, as are the one-chains $\lambda$ and $\mu$ and the positive integer $p$.

\begin{remark}
Note that the above three cases, plus the case of $S^3$, provide certificates for all lens spaces.
\end{remark}

\subsection{Prism non-lens spaces}

The certificate:
\begin{enumerate}
\item [(1)] a triangulation $\cover{\calT}$ of a three-manifold $\cover{M}$;
\item [(2)] integers $p, q, r, s$ where $p \geq 1$ and $q > 1$;
\item [$\ast$(3)] a certificate that $\cover{M}$ is the lens space $L(2pq, ps+qr)$;
\item [(4)] a simplicial involution $\phi$ of $\cover{\calT}$ that has no fixed points;
\item [(5)] a simplicial isomorphism between $\calT$ and $\cover{\calT} / \langle \phi \rangle$.
\end{enumerate}

To verify this certificate we must check the following:
\begin{enumerate}
\item that $ps-qr = 1$;
\item that the given certificate establishes that $\cover{M}$ is the lens space $L(2pq, ps+qr)$;
\item that the subgroup of $H_1(\cover{M})$ that is fixed by $\phi$ has order $2p$;
\item that $\phi$ has no fixed points;
\item that the given map between $\calT$ and $\cover{\calT} / \langle \phi \rangle$ is a simplicial isomorphism.
\end{enumerate}

\refprop{PrismCriterion} then implies that $M$ is the prism manifold $P(p,q)$. 
This has two possible Seifert fibrations. 
The algorithm outputs the one with spherical base space, which has Seifert data $(0, 1/2, -1/2, q/p)$.

We need to show that if $M$ is the manifold $P(p,q)$, then there is a certificate as above that can be verified in polynomial time.
\refprop{PrismCriterion} gives that the prism manifold $P(p,q)$ has a double cover $\cover{M}$ that is the lens space $L(2pq, ps + qr)$.
The triangulation $\calT$ lifts to a triangulation $\cover{\calT}$ for $\cover{M}$.
As explained in the previous subsections, there is a certificate that establishes that $\cover{M}$ is $L(2pq, ps+qr)$.
As usual, $2pq$ has at most polynomially many digits (as a function of $|\cover{\calT}|$) as do $p$ and $q$. 
Now, $s/r$ can be any rational number so that $ps - qr = 1$.
However, we may rechoose $r$ and $s$ so that $0 \leq r \leq p$ and hence $|s| \leq |ps| = |1 + qr|$.
Thus, it can be verified in polynomial time that $ps - qr = 1$.
The remaining parts of the certificate may be checked in polynomial time.
In particular, one uses linear algebra to verify that the subgroup of $H_1(\cover{M})$ that is fixed by the covering involution $\phi$ has order $2p$.

Again certain parts of this certificate are not strictly necessary. 
Since the first homology of a prism manifold is generated by at most two elements, it has at most three double covers. 
For each such cover, one can compute its lifted triangulation and its isomorphism signature. One of
these triangulations will be $\cover{\calT}$. 
If we proceed in this way, we do not need to be provided with (4) and (5) in the certificate. 
The order of the subgroup of $H_1(\cover{M})$, fixed by the covering transformation, can be computed in polynomial time; 
this gives $p$. 
Since we have established that lens space recognition is in \FNP, when this is applied to $\cover{M}$, the integers $2pq$ and $ps+qr$ are given as the lens space coefficients; 
hence these do not need to be provided as part of a certificate. 
Once we have $p$ and $2pq$, we can compute $q$. 
As observed above, $s$ and $r$ can be any integers satisfying $ps - qr = 1$ and hence can be computed using the Euclidean algorithm, and can be arranged so that $0 \leq r \leq p$ and $|s| \leq |1 + qr|$. 
Then $ps + qr$ can be computed and we can check whether it equals the second integer that is output from the recognition of $\cover{M}$ as a lens space.

\subsection{Platonic manifolds}

The certificate:
\begin{enumerate}
\item [$\ast$(1)] a triangulation $\cover{\calT}$ of a three-manifold $\cover{M}$;
\item [$\ast$(2)] a group $\calP$ of simplicial isomorphisms of $\cover{\calT}$ that acts freely;
\item [(3)]
a simplicial isomorphism between $\cover{\calT}/\calP$ and $\calT$;
\item [(4)] an isomorphism between $\calP$ and one of the platonic groups $A_4$, $S_4$ or $A_5$;
\item [$\ast$(5)] a certificate that certifies that $\cover{M}$ is a lens space using one of the certificates described above in Sections \ref{Sec:RP3}, \ref{Sec:NonPrismNonRP3Lens} or \ref{Sec:PrismLens}.
\end{enumerate}

To verify this certificate we must check the following:
\begin{enumerate}
\item
that the given isomorphism between $\calP$ and one of the platonic groups is indeed an isomorphism;
\item
that $\calP$ acts freely on $\cover{\calT}$;
\item
that the given simplicial isomorphism between $\cover{\calT}/\calP$ and $\calT$ is indeed a simplicial isomorphism;
\item that the given certificate establishes that $\cover{M}$ is homeomorphic to a lens space;
\item the resulting Seifert invariants of $M$, using \refprop{PlatonicHomeo}.
\end{enumerate}

It is clear that if there is such a certificate, then $M$ is regularly covered by a lens space, with deck group given by the platonic group $\calP$.
Hence, by \refprop{PlatonicCriterion}, $M$ is a platonic manifold with type $\calP$.
Conversely, if $M$ is such a manifold, then by \refprop{PlatonicCriterion}, there is such a finite cover, and hence the certificate exists.
It can be verified in polynomial time.

It is clear that (3) and (4) are not actually needed as part of the certificate. The simplicial isomorphism in (3) can be computed using Burton's algorithm and it can readily be checked that the given group $\calP$ of covering transformations is isomorphic to one of $A_4$, $S_4$ or $A_5$.

An alternative certificate, instead of giving (1) and (2), could give a surjective homomorphism from $\pi_1(M)$ to $\calP$, the desired platonic group.  
The verifier could then build the cover $\cover{\calT}$ and its deck group directly, in polynomial time.

\bibliographystyle{plainurl}
\bibliography{lens_spaces}

\begin{thebibliography}{10}

\bibitem{AschenbrennerFriedlWilton}
Matthias Aschenbrenner, Stefan Friedl, and Henry Wilton.
\newblock Decision problems for 3-manifolds and their fundamental groups.
\newblock In {\em Interactions between low-dimensional topology and mapping
  class groups}, volume~19 of {\em Geom. Topol. Monogr.}, pages 201--236. Geom.
  Topol. Publ., Coventry, 2015.
\newblock \href {http://arxiv.org/abs/1405.6274} {\path{arXiv:1405.6274}},
  \href {https://doi.org/10.2140/gtm.2015.19.201}
  {\path{doi:10.2140/gtm.2015.19.201}}.

\bibitem{BonahonOtal83}
Francis Bonahon and Jean-Pierre Otal.
\newblock Scindements de {H}eegaard des espaces lenticulaires.
\newblock {\em Ann. Sci. \'{E}cole Norm. Sup. (4)}, 16(3):451--466 (1984),
  1983.
\newblock \href {https://doi.org/10.24033/asens.1455}
  {\path{doi:10.24033/asens.1455}}.

\bibitem{BredonWood69}
Glen~E. Bredon and John~W. Wood.
\newblock Non-orientable surfaces in orientable {$3$}-manifolds.
\newblock {\em Invent. Math.}, 7:83--110, 1969.
\newblock \href {https://doi.org/10.1007/BF01389793}
  {\path{doi:10.1007/BF01389793}}.

\bibitem{Burton11}
Benjamin~A. Burton.
\newblock The {P}achner graph and the simplification of 3-sphere
  triangulations.
\newblock In {\em Computational geometry ({SCG}'11)}, pages 153--162. ACM, New
  York, 2011.
\newblock \href {http://arxiv.org/abs/1011.4169} {\path{arXiv:1011.4169}},
  \href {https://doi.org/10.1145/1998196.1998220}
  {\path{doi:10.1145/1998196.1998220}}.

\bibitem{GeigesThies23}
Hansj\"{o}rg Geiges and Norman Thies.
\newblock Klein bottles in lens spaces.
\newblock {\em Involve}, 16(4):621--636, 2023.
\newblock \href {http://arxiv.org/abs/2205.04909} {\path{arXiv:2205.04909}},
  \href {https://doi.org/10.2140/involve.2023.16.621}
  {\path{doi:10.2140/involve.2023.16.621}}.

\bibitem{Hempel}
John Hempel.
\newblock {\em 3-manifolds}.
\newblock AMS Chelsea Publishing, Providence, RI, 2004.
\newblock Reprint of the 1976 original.
\newblock \href {https://doi.org/10.1090/chel/349}
  {\path{doi:10.1090/chel/349}}.

\bibitem{Haraway20}
Robert~Haraway III.
\newblock Determining hyperbolicity of compact orientable 3-manifolds with
  torus boundary.
\newblock {\em J. Comput. Geom.}, 11(1):125--136, 2020.
\newblock \href {http://arxiv.org/abs/1410.7115} {\path{arXiv:1410.7115}},
  \href {https://doi.org/10.20382/jocg.v11i1a5}
  {\path{doi:10.20382/jocg.v11i1a5}}.

\bibitem{HarawayHoffman19}
Robert~Haraway III and Neil~R Hoffman.
\newblock On the complexity of cusped non-hyperbolicity, 2019.
\newblock \href {http://arxiv.org/abs/1907.01675} {\path{arXiv:1907.01675}}.

\bibitem{Ivanov08}
S.~V. Ivanov.
\newblock The computational complexity of basic decision problems in
  $3$--dimensional topology.
\newblock {\em Geom. Dedicata}, 131:1--26, 2008.
\newblock \href {https://doi.org/10.1007/s10711-007-9210-4}
  {\path{doi:10.1007/s10711-007-9210-4}}.

\bibitem{KannanBachem79}
Ravindran Kannan and Achim Bachem.
\newblock Polynomial algorithms for computing the {S}mith and {H}ermite normal
  forms of an integer matrix.
\newblock {\em SIAM J. Comput.}, 8(4):499--507, 1979.
\newblock \href {https://doi.org/10.1137/0208040} {\path{doi:10.1137/0208040}}.

\bibitem{Kneser28}
H.~Kneser.
\newblock Geschlossene {Fl{\"a}chen} in dreidimensionalen {Mannigfaltigkeiten}.
\newblock {\em Naturwiss.}, 16:973, 1928.

\bibitem{Kuperberg18}
Greg Kuperberg.
\newblock Identifying lens spaces in polynomial time.
\newblock {\em Algebr. Geom. Topol.}, 18(2):767--778, 2018.
\newblock \href {http://arxiv.org/abs/1509.02887} {\path{arXiv:1509.02887}},
  \href {https://doi.org/10.2140/agt.2018.18.767}
  {\path{doi:10.2140/agt.2018.18.767}}.

\bibitem{Kuperberg19}
Greg Kuperberg.
\newblock Algorithmic homeomorphism of 3-manifolds as a corollary of
  geometrization.
\newblock {\em Pacific J. Math.}, 301(1):189--241, 2019.
\newblock \href {http://arxiv.org/abs/1508.06720} {\path{arXiv:1508.06720}},
  \href {https://doi.org/10.2140/pjm.2019.301.189}
  {\path{doi:10.2140/pjm.2019.301.189}}.

\bibitem{Lackenby:Composite}
Marc Lackenby.
\newblock The crossing number of composite knots.
\newblock {\em J. Topol.}, 2(4):747--768, 2009.
\newblock \href {http://arxiv.org/abs/0805.4706} {\path{arXiv:0805.4706}},
  \href {https://doi.org/10.1112/jtopol/jtp028}
  {\path{doi:10.1112/jtopol/jtp028}}.

\bibitem{Lackenby:CoreCurves}
Marc Lackenby.
\newblock Core curves of triangulated solid tori.
\newblock {\em Trans. Amer. Math. Soc.}, 366(11):6027--6050, 2014.
\newblock \href {http://arxiv.org/abs/1106.2934} {\path{arXiv:1106.2934}},
  \href {https://doi.org/10.1090/S0002-9947-2014-06170-9}
  {\path{doi:10.1090/S0002-9947-2014-06170-9}}.

\bibitem{Lackenby:Knottedness}
Marc Lackenby.
\newblock The efficient certification of knottedness and {T}hurston norm.
\newblock {\em Adv. Math.}, 387:Paper No. 107796, 142, 2021.
\newblock \href {http://arxiv.org/abs/1604.00290} {\path{arXiv:1604.00290}},
  \href {https://doi.org/10.1016/j.aim.2021.107796}
  {\path{doi:10.1016/j.aim.2021.107796}}.

\bibitem{LackenbySchleimer12}
Marc Lackenby and Saul Schleimer.
\newblock Lens space recognition is in {\NP}.
\newblock In {\em Oberwolfach Rep.}, volume~9, pages 1421--1424. European
  Mathematical Society, 2012.

\bibitem{Livesay60}
G.~R. Livesay.
\newblock Fixed point free involutions on the {$3$}-sphere.
\newblock {\em Ann. of Math. (2)}, 72:603--611, 1960.
\newblock \href {https://doi.org/10.2307/1970232} {\path{doi:10.2307/1970232}}.

\bibitem{Myers81}
Robert Myers.
\newblock Free involutions on lens spaces.
\newblock {\em Topology}, 20(3):313--318, 1981.
\newblock \href {https://doi.org/10.1016/0040-9383(81)90005-7}
  {\path{doi:10.1016/0040-9383(81)90005-7}}.

\bibitem{Orlik72}
Peter Orlik.
\newblock {\em Seifert manifolds}.
\newblock Lecture Notes in Mathematics, Vol. 291. Springer-Verlag, Berlin-New
  York, 1972.
\newblock \href {https://doi.org/10.1007/BFb0060329}
  {\path{doi:10.1007/BFb0060329}}.

\bibitem{Perelman1}
Grisha Perelman.
\newblock The entropy formula for the {R}icci flow and its geometric
  applications, 2002.
\newblock \href {http://arxiv.org/abs/math/0211159}
  {\path{arXiv:math/0211159}}.

\bibitem{Perelman3}
Grisha Perelman.
\newblock Finite extinction time for the solutions to the {R}icci flow on
  certain three-manifolds, 2003.
\newblock \href {http://arxiv.org/abs/math/0307245}
  {\path{arXiv:math/0307245}}.

\bibitem{Perelman2}
Grisha Perelman.
\newblock Ricci flow with surgery on three-manifolds, 2003.
\newblock \href {http://arxiv.org/abs/math/0303109}
  {\path{arXiv:math/0303109}}.

\bibitem{Reidemeister35}
Kurt Reidemeister.
\newblock Homotopieringe und {L}insenr\"{a}ume.
\newblock {\em Abh. Math. Sem. Univ. Hamburg}, 11(1):102--109, 1935.
\newblock \href {https://doi.org/10.1007/BF02940717}
  {\path{doi:10.1007/BF02940717}}.

\bibitem{Schleimer11}
Saul Schleimer.
\newblock Sphere recognition lies in {NP}.
\newblock In {\em Low-dimensional and symplectic topology}, volume~82 of {\em
  Proc. Sympos. Pure Math.}, pages 183--213. Amer. Math. Soc., Providence, RI,
  2011.
\newblock \href {http://arxiv.org/abs/math/0407047}
  {\path{arXiv:math/0407047}}, \href {https://doi.org/10.1090/pspum/082}
  {\path{doi:10.1090/pspum/082}}.

\bibitem{Schleimer01}
Saul~David Schleimer.
\newblock {\em Almost normal {H}eegaard splittings}.
\newblock PhD thesis, University of California, Berkeley, 2001.
\newblock \url{http://sschleimer.warwick.ac.uk/Maths/thesis.pdf}.

\bibitem{ScottShort}
Peter Scott and Hamish Short.
\newblock The homeomorphism problem for closed 3-manifolds.
\newblock {\em Algebr. Geom. Topol.}, 14(4):2431--2444, 2014.
\newblock \href {http://arxiv.org/abs/1211.0264} {\path{arXiv:1211.0264}},
  \href {https://doi.org/10.2140/agt.2014.14.2431}
  {\path{doi:10.2140/agt.2014.14.2431}}.

\bibitem{Stocking}
Michelle Stocking.
\newblock Almost normal surfaces in {$3$}-manifolds.
\newblock {\em Trans. Amer. Math. Soc.}, 352(1):171--207, 2000.
\newblock \href {https://doi.org/10.1090/S0002-9947-99-02296-5}
  {\path{doi:10.1090/S0002-9947-99-02296-5}}.

\bibitem{SeifertThrelfall33}
W.~Threlfall and H.~Seifert.
\newblock Topologische {U}ntersuchung der {D}iskontinuit\"atsbereiche endlicher
  {B}ewegungsgruppen des dreidimensionalen sph\"arischen {R}aumes ({S}chlu\ss).
\newblock {\em Math. Ann.}, 107(1):543--586, 1933.
\newblock \href {https://doi.org/10.1007/BF01448910}
  {\path{doi:10.1007/BF01448910}}.

\bibitem{Thurston97}
William~P. Thurston.
\newblock {\em Three-dimensional geometry and topology. {V}ol. 1}, volume~35 of
  {\em Princeton Mathematical Series}.
\newblock Princeton University Press, Princeton, NJ, 1997.
\newblock Edited by Silvio Levy.
\newblock \href {https://doi.org/10.1515/9781400865321}
  {\path{doi:10.1515/9781400865321}}.

\end{thebibliography}
\end{document}